\documentclass[12pt,a4paper]{article}

\usepackage{amsmath,amsfonts,amssymb,amsthm}
\usepackage{mathrsfs}
\usepackage{enumitem}
\usepackage{bbold}
\usepackage{graphicx}
\usepackage{a4wide}
\usepackage{authblk}
\usepackage{cite}
\allowdisplaybreaks

\let\scr\mathscr

\expandafter\ifx\csname pdfglyphtounicode\endcsname\relax\else
\expandafter\ifx\csname pdfgentounicode\endcsname\relax\else
\input glyphtounicode.tex
\fi
\fi

\renewcommand{\leq}{\leqslant}

\renewcommand{\geq}{\geqslant}

\makeatletter
\newcommand*{\bigcorr@macro}[2]{\sbox{0}{\mbox{$#1($}}\dimen0=\ht0
                \advance\dimen0 by \dp0
                \multiply\dimen0 by #2 \divide\dimen0 by 100}
\newcommand*{\bigcorr@big}[2]{\mbox{$#1\left#2\bigcorr@macro{#1}{85}\vrule
                   height \dimen0 depth 0pt width 0pt\right.\n@space$}}
\newcommand*{\bigcorr@Big}[2]{\mbox{$#1\left#2\bigcorr@macro{#1}{115}\vrule
                   height \dimen0 depth 0pt width 0pt\right.\n@space$}}
\newcommand*{\bigcorr@bigg}[2]{\mbox{$#1\left#2\bigcorr@macro{#1}{145}\vrule
                   height \dimen0 depth 0pt width 0pt\right.\n@space$}}
\newcommand*{\bigcorr@Bigg}[2]{\mbox{$#1\left#2\bigcorr@macro{#1}{175}\vrule
                   height \dimen0 depth 0pt width 0pt\right.\n@space$}}
\DeclareRobustCommand*{\big}[1]{{\mathpalette\bigcorr@big{#1}}}
\DeclareRobustCommand*{\Big}[1]{{\mathpalette\bigcorr@Big{#1}}}
\DeclareRobustCommand*{\bigg}[1]{{\mathpalette\bigcorr@bigg{#1}}}
\DeclareRobustCommand*{\Bigg}[1]{{\mathpalette\bigcorr@Bigg{#1}}}

\DeclareRobustCommand*{\bil}[1]{\mathopen{\bi{#1}}}

\DeclareRobustCommand*{\bir}[1]{\mathclose{\bi{#1}}}
\makeatother

\newtheorem{lemma}{Lemma}

\newtheorem{theorem}{Theorem}
\newtheorem{definition}{Definition}
\theoremstyle{remark}

\def\BB{\mathbb{B}}

\def\DD{\mathbb{D}}

\def\KK{\mathbb{K}}
\def\SS{\mathbb{S}}

\def\UU{\mathbb{U}}

\def\1{\mbox{1\hspace{-.25em}I}}
\newcommand{\Liminf}{\mathop{\underline{\lim}}\limits}

\newcommand{\Ex}{{\mathbf{E}}}
\newcommand{\Pb}{{\mathbf{P}}}

\let\bil\mathopen
\let\bir\mathclose

\begin{document}
\title{Localization of two radioactive sources on the plane}
\author[1]{O.V.~Chernoyarov}
\author[2]{S.~Dachian\thanks{Corresponding author: Serguei.Dachian@univ-lille.fr}}
\author[3]{C.~Farinetto}
\author[4]{Yu.A.~Kutoyants}
\affil[1,4]{\small National Research University ``MPEI'', Moscow, Russia}
\affil[2]{\small University of Lille, Lille, France}
\affil[3,4]{\small Le Mans University, Le Mans, France}
\affil[1,4]{\small Tomsk State University, Tomsk, Russia}

\date{}
\maketitle

\begin{abstract}
The problem of localization on the plane of two radioactive sources by $K$
detectors is considered.  Each detector records a realization of an
inhomogeneous Poisson process whose intensity function is the sum of signals
arriving from the sources and of a constant Poisson noise of known intensity.
The time of the beginning of emission of the sources is known, and the main
problem is the estimation of the positions of the sources.  The properties of
the maximum likelihood and Bayesian estimators are described in the
asymptotics of large signals in three situations of different regularities of
the fronts of the signals: smooth, cusp-type and change-point type.

\end{abstract}
\noindent MSC 2000 Classification: 62M02,  62G10, 62G20.

\noindent {\sl Key words}: \textsl{Poisson process, parameter
  estimation, cusp-type singularity.}

\section{Introduction}

Suppose that there are two radioactive sources and $K$ detectors on the plane.
The sources start emitting at a known time, which can be taken $t=0$ without
loss of generality.  The detectors receive Poisson signals with additive
noise.  The intensity functions of these processes depend on the positions of
the detectors (known) and the positions of the sources (unknown), and the main
problem is the estimation of the positions of the sources.  An example of a
possible configuration of the sources and of the detectors on the plane is
given in~Fig.~\ref{2sour}.

\begin{figure}[!ht]
\centering
\includegraphics[width=0.5\textwidth]{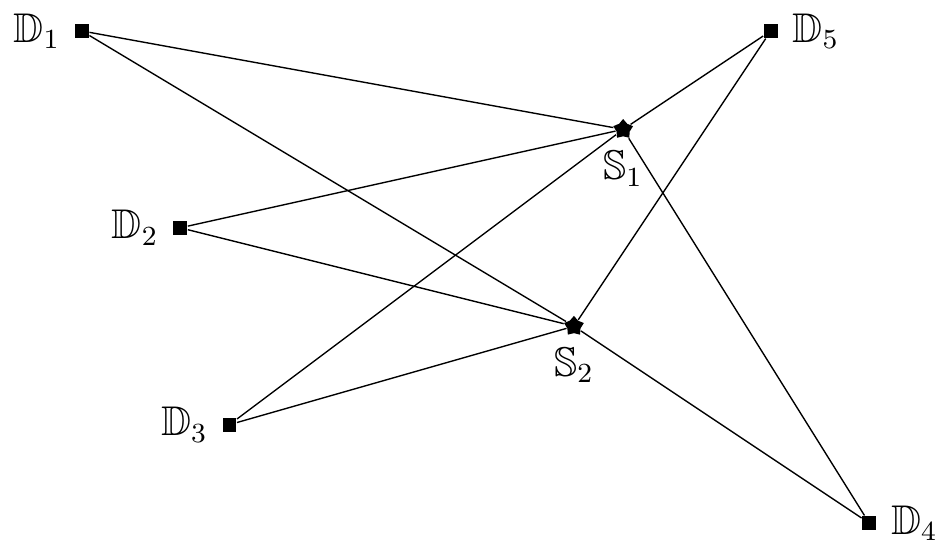}
\caption{Model of observations: $\SS_1$, $\SS_2$ are the positions of the
  sources and $\DD_i$, $i=1,\ldots,5$, are the positions of the sensors}
\label{2sour}
\end{figure}

Due to their practical importance, the problems of localization of the sources
with Poisson, Gaussian or more general classes of distributions are widely
studied in engineering literature, see,
e.g.,~\cite{BMBD15,CHY02,DLF12,Pu09,St10,W15,Kn10,WZZ}, as well as the
Handbook~\cite{ZB19} and the references therein.  Mathematical statements are
less known.  This work is a continuation of a study initiated in~\cite{FKT20}
and then developed in~\cite{AK19,CDFK22,CDK20,CK20,Kut20}, where it was always
supposed that there is only one source on the plane.  In the
works~\cite{CDK20,CK20,FKT20} it was supposed that the moment of the beginning
of the emission $\tau_0$ is known and the unknown parameter was just the
position of the source.  The case where $\tau_0$ is unknown too and we have to
estimate both~$\tau_0$ and the position of the source was treated
in~\cite{AK19,CDFK22}.  In all these works the properties of the maximum
likelihood estimator (MLE) and of the Bayesian estimators (BEs) of the
position (or of $\tau_0$ and of the position) where described.  Moreover, the
properties of the least squares estimators where equally described.  The
inhomogeneous Poisson processes and the diffusion processes were considered as
models of observations.  The properties of the MLE and of the BEs were
described in the asymptotics of large signals or in the asymptotics of small
noise.

In the present work we suppose that there are $K$ detectors and two
radioactive sources emitting signals which can be described as inhomogeneous
Poisson processes and the emission starts at the (known) moment $t=0$.  So, we
need to estimate the positions of the sources only.  As in all preceding
works, the properties of the estimators are described with the help of the
Ibragimov-Khasminskii approach (see~\cite{IH81}), which consists in the
verification of certain properties of the normalized likelihood ratio process.

The information about the positions of the sources is contained in the times
of arrival of the signals to the detectors.  These times depend on the
distances between the sources and the detectors.  The estimation of the
positions depend on the estimation of these times, and here the form of the
fronts of the arriving signals plays an important role.  We consider three
types of fronts: smooth, cusp-type and change-point type.  In the smooth case
the Fisher information matrix is finite and the estimators are asymptotically
normal.  In the cusp-case the fronts are described by a function which is
continuous, but the Fisher information does not exist (is infinite).  For
Poisson processes, statisticaal problems with such type singularities was
first considered in~\cite{D03}.  In this case the Bayesian estimators converge
to a random vector defined with the help of some functionals of fractional
Brownian motions.  In the change-point case with discontinuous intensities the
limit distribution of Bayesian estimators is defined with the help of some
Poisson processes.  In all the three cases we discuss the asymptotic
efficiency of the proposed estimators.

Special attention is paid to the condition of identifiability, i.e., to the
description of the admissible configurations of detectors which allow the
consistent estimation of the positions of the sources.  It is shown that if
the detectors do not lay on a cross, then it is impossible to find two
different pairs of sources which provide the same moments of arrival to the
detectors and hence the consistent estimation is possible.

\section{Main results}

\subsection{Model of observations}

We suppose that there are $K$ detectors $\DD_1,\ldots,\DD_K$ located on the
plane at the (known) points $D_k=(x_k,y_k)^\top$, $k=1,\ldots,K$ ($K\geq 4$),
and two sources $\SS_1$ and $\SS_2$ located at the (unknown) points
$S_1=(x_1',y_1')^\top$ and $S_2=(x_2',y_2')^\top$.  Therefore, the unknown
parameter is $\vartheta =(x_1',y_1',x_2',y_2')^\top$, but it will be
convenient to write it as $\vartheta = (\vartheta_1,\ldots,\vartheta
_4)^\top\in\Theta$ and as $\vartheta = (\vartheta^{(1)},\vartheta^{(2)})^\top$
with obvious notations.  For the true value of $\vartheta$ we will often use
the notations $\vartheta_0 = (\vartheta_0^{(1)},\vartheta_0^{(2)})^\top$,
$\vartheta_0^{(1)} = (\vartheta_1^\circ,\vartheta_2^\circ)^\top =
(x_1'^\circ,y_1'^\circ)^\top$, $\vartheta_0^{(2)} =
(\vartheta_3^\circ,\vartheta_4^\circ)^\top = (x_2'^\circ,y_2'^\circ)^\top$.
The set $\Theta$ is an open, bounded, convex subset of~$\mathcal{R}^4$.  We
suppose, of course, that the positions of the detectors are all different.  We
suppose as well that a position of a source does not coincide with a position
of a detector.

The sources start emitting at the moment $t=0$.  The $k$-th detector records a
realization $X_k=(X_k(t),\ 0\leq t\leq T)$ of an inhomogeneous Poisson process
of intensity function
\[
\lambda _{k,n}(\vartheta,t) = n S_{1,k}(\vartheta^{(1)},t) + n
S_{2,k}(\vartheta^{(2)},t) + n\lambda_0,\qquad 0\leq t\leq T,
\]
where $n\lambda_0>0$ is the intensity of the Poisson noise and $n
S_{i,k}(\vartheta^{(i)},t)$ is the signal recorded from the $i$-th source,
$i=1,2$.  We suppose that the recorded signals have the following structure
\begin{equation}
\label{signals}
S_{i,k}(\vartheta^{(i)},t) = \psi\bigl(t-\tau_{k}(\vartheta^{(i)})\bigr)
S_{i,k}(t),
\end{equation}
where $S_{i,k}(t)>0 $ is a bounded function and $\tau_{k}(\vartheta^{(i)})$ is
the time of the arrival of the signal from the $i$-th source to the $k$-th
detector, i.e.,
\[
\tau_{k}(\vartheta^{(i)}) = \nu^{-1} {\bil\|D_k-S_i\bir\|}_2 = \nu^{-1}
\bigl(\bil|x_k-x_i'\bir|^2+\bil|y_k-y_i'\bir|^2\bigr)^{1/2} ,\qquad i=1,2,
\]
and $\nu>0$ is the rate of the propagation of the signals.  Here and in the
sequel, we denote ${\bil\|\cdot\bir\|}_2$ and ${\bil\|\cdot\bir\|}_4$ the
Euclidean norms in $\mathcal{R}^2$ and $\mathcal{R}^4$ respectively.

The function $\psi(\cdot)$ describes the fronts of the signals.  As in our
preceding works (see, e.g.,~\cite{CDFK22}) we take
\[
\psi(s)=\Bigl|\frac{s}{\delta}\Bigr|^\kappa \1_{\{0\leq s\leq \delta\}}
+\1_{\{s>\delta\}},
\]
where $\delta>0$ and the parameter $\kappa \geq 0$ describes the regularity of
the statistical problem.  If~$\kappa \geq \frac{1}{2}$, we have a regular
statistical experiment, if $\kappa \in \bigl(0,\frac{1}{2}\bigr)$, we have a
singularity of cusp type, and if $\kappa=0$, the intensity is a discontinuous
function and we have a change-point model.  The examples of these three cases
are given in Fig.~\ref{F1}, where we put $n S_{i,k}(t)\equiv 2$,
$n\lambda_0=1$ and \textbf{a)} $\kappa=1$, \textbf{b)} $\kappa=1/4$,
\textbf{c)} $\kappa=0$.

\begin{figure}[!ht]
\centering
\includegraphics[width=0.7\textwidth]{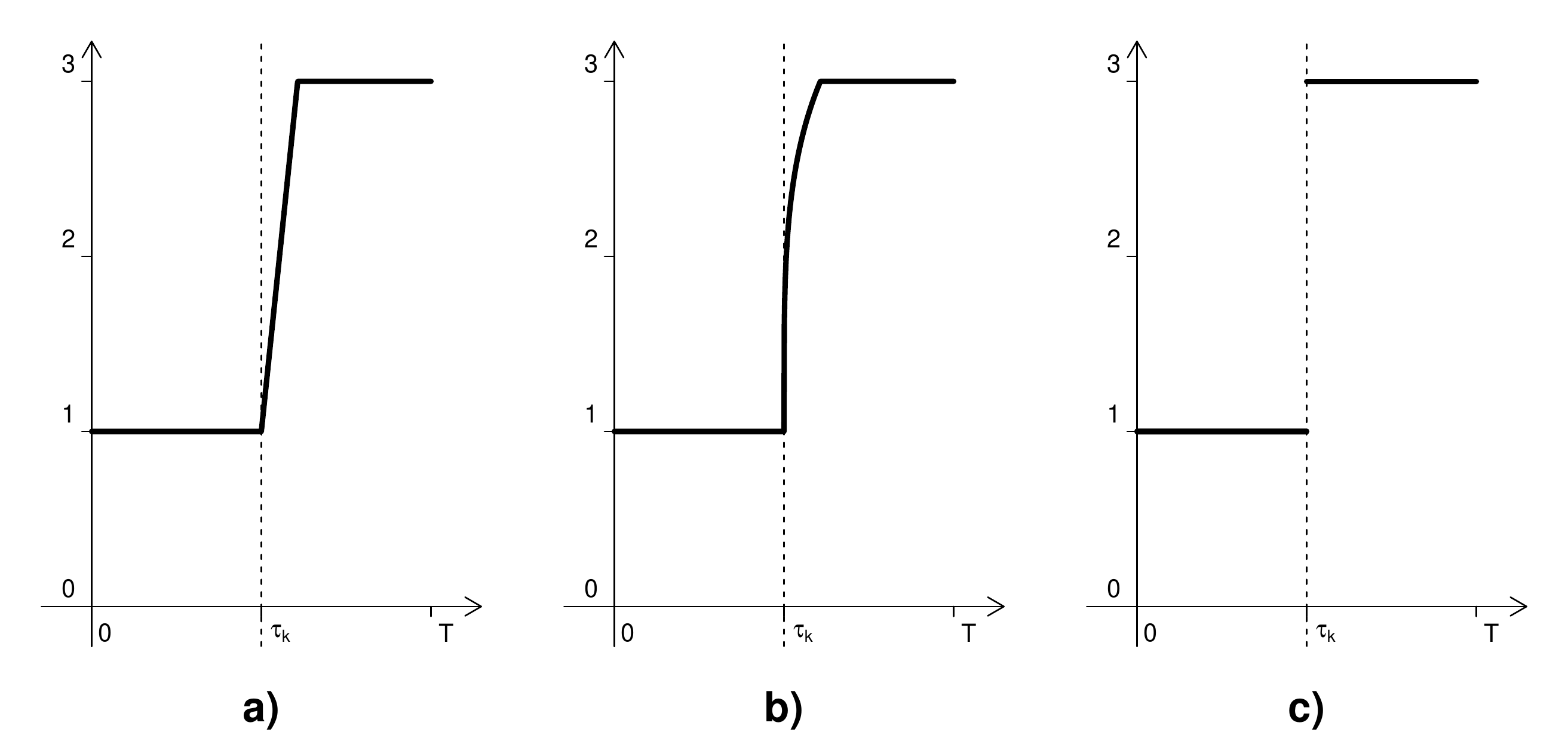}
\caption{Intensities with three types of fronts of arriving signals}
\label{F1}
\end{figure}
  
Note that $\psi(s)=0$ for $s<0$, and therefore for
$t<\tau_{k}\bigl(\vartheta^{(1)}\bigr) \wedge
\tau_{k}\bigl(\vartheta^{(2)}\bigr)$ the intensity function is
$\lambda_{k,n}(\vartheta,t)=n\lambda_0$.  According to the form of the
intensity function, all the information concerning the positions of the
sources is clearly contained in the moments of arrival
$\tau_{k}(\vartheta^{(i)})$, $i=1,2$, $k=1,\ldots,K$.

We are interested in the situation where the errors of estimation are small.
In the problem of localization it is natural to suppose that the registered
intensities take large values.  Therefore we study the properties of the
estimators of the positions in the asymptotics of large intensities, that is
why we introduce in the intensity functions the factor $n$ and the asymptotics
corresponds to the limit $n\rightarrow\infty$.  In Section 3 we explain that
the condition $K\geq 4$ is necessary for the existence of consistent
estimators.

\subsection{Maximum likelihood and Bayesian estimators}

As the intensity functions $\lambda _{k,n}\left(\cdot \right)$ are bounded and
separated from zero, the measures $\Pb_\vartheta^{(n)}$, $\vartheta\in\Theta$,
induced by the observations $X^K = (X_1,\ldots,X_K)$ on the space of the
realizations are equivalent, and the likelihood ratio function is
\[
L\bigl(\vartheta,X^K\bigr) = \exp\Biggl\{\sum_{k=1}^{K}\int_{0}^{T}
\ln\frac{\lambda_{k,n}(\vartheta,t)}{n\lambda_0}\,\mathrm{d}X_k(t) -
\sum_{k=1}^{K}\int_{0}^{T} [\lambda_{k,n}(\vartheta,t) -
  n\lambda_0]\mathrm{d}t\Biggr\}, \quad \vartheta \in\Theta
\]

The MLE $\hat\vartheta_n$ and the BEs for quadratic loss function
$\tilde\vartheta_n$ are defined by the usual relations
\[
L\bigl(\hat\vartheta_n,X^K\bigr) = \sup_{\vartheta\in\Theta}
L\bigl(\vartheta,X^K\bigr),\qquad \qquad \tilde\vartheta_n =
\frac{\int_{\Theta} \vartheta\, p(\vartheta)
  L\bigl(\vartheta,X^K\bigr)\,\mathrm{d}\vartheta}{\int_{\Theta} p(\vartheta)
  L\bigl(\vartheta,X^K\bigr)\,\mathrm{d}\vartheta}\,,
\]
where $p(\vartheta)$, $\vartheta\in\Theta$, is a strictly positive and
continuous \emph{a priori} density of the (random) parameter $\vartheta$.

\subsubsection{Smooth fronts}

First we consider the regular case in a slightly different setup, which can be
seen as more general.  The intensities of the observed processes are supposed
to be
\[
\lambda_{k,n}(\vartheta,t) = n S_{1,k}(\vartheta^{(1)},t) + n
S_{2,k}(\vartheta^{(2)},t) + n\lambda_0 = n\lambda_{k}(\vartheta,t),\qquad
\quad 0\leq t\leq T,
\]
where $\lambda_{k}(\vartheta,t)$ is defined by the last equality and
$S_{i,k}(\vartheta^{(i)},t) = s_{i,k}\bigl(t-\tau_k(\vartheta^{(i)})\bigr)$,
$i=1,2$, $k=1,\ldots,K$.

For the derivatives we have the expressions
\[
\frac{\partial s_{1,k}\bigl(t-\tau_k(\vartheta^{(1)})\bigr)}{\partial
  \vartheta _1} = \nu ^{-1}
s_{1,k}'\bigl(t-\tau_k(\vartheta^{(1)})\bigr)\,\frac{x_k-x_1'}{\rho_{1,k}} =
\nu ^{-1} s_{1,k}'\bigl(t-\tau_k(\vartheta^{(1)})\bigr)\, \cos(\alpha_{1,k}),
\]
where $\rho_{1,k}={\bil\|D_k-S_1\bir\|}_2$ and
$\cos(\alpha_{1,k})=(x_k-x_1')\rho_{1,k}^{-1}$.  Similarly, we obtain
\[
\frac{\partial s_{1,k}\bigl(t-\tau_k(\vartheta^{(1)})\bigr)}{\partial
  \vartheta_2} = \nu^{-1} s_{1,k}'\bigl(t-\tau_k(\vartheta^{(1)})\bigr)\,
\sin(\alpha_{1,k})
\]
and, of course, we have $\partial s_{1,k}/\partial \vartheta _3 = \partial
s_{1,k}/\partial \vartheta _4=0$.  Let us recall here that we use the
notations $\vartheta = (\vartheta^{(1)},\vartheta^{(2)})^\top$,
$\vartheta^{(1)} = (\vartheta_1,\vartheta_2)^\top$ and
$\vartheta^{(2)}=(\vartheta_3,\vartheta_4)^\top$.  For $\partial
s_{2,k}/\partial \vartheta_3$ and $\partial s_{2,k}/\partial \vartheta_4$ we
obtain similar expressions:
\begin{align*}
\frac{\partial s_{2,k}\bigl(t-\tau_k(\vartheta^{(2)})\bigr)}{\partial
  \vartheta_3} &= \nu ^{-1} s_{2,k}' \bigl(t-\tau_k(\vartheta^{(2)})\bigr)\,
\cos(\alpha_{2,k}),\\*
\frac{\partial s_{2,k}\bigl(t-\tau_k(\vartheta^{(2)})\bigr)}{\partial
  \vartheta_4} &= \nu ^{-1} s_{2,k}' \bigl(t-\tau_k(\vartheta^{(2)})\bigr)\,
\sin(\alpha_{2,k}).
\end{align*}

Let us now introduce the two vectors $m_{1,k} = m_{1,k}(\vartheta^{(1)}) =
\bigl(\cos(\alpha_{1,k}),\sin(\alpha_{1,k})\bigr)^\top$ and $m_{2,k} =
m_{2,k}(\vartheta^{(2)}) =
\bigl(\cos(\alpha_{2,k}),\sin(\alpha_{2,k})\bigr)^\top$.  We have
${\bil\|m_{i,k}\bir\|}_2=1$, $i=1,2$.  We will several times use the expansion
(below $u^{(1)}=(u_1,u_2)^\top$, $u^{(2)}=(u_3,u_4)^\top$ and $\varphi_n
\rightarrow 0$)
\[
\tau_{k}(\vartheta_0 ^{(i)}+u^{(i)}\varphi _n) = \tau_{k}(\vartheta_0^{(i)}) -
\nu^{-1} \langle m_{i,k}^\circ,u^{(1)}\rangle\, \varphi_n + O(\varphi_n^2),
\]
where $m_{i,k}^\circ=m_{i,k}(\vartheta_0^{(i)}) =
\bigl(\cos(\alpha_{1,k}^\circ),\sin(\alpha_{1,k}^\circ)\bigr)^\top$.

For simplicity of exposition we will sometimes use the notations $\tau_{1,k} =
\tau_k(\vartheta^{(1)})$ and $\tau_{2,k} = \tau_k(\vartheta^{(2)})$.
 
The Fisher information matrix is
\[
\mathrm{I}(\vartheta)_{4\times 4} = \sum_{k=1}^{K}\int_{0}^{T}
\frac{\dot\lambda_{k}(\vartheta,t)\,
  \dot\lambda_{k}(\vartheta,t)^\top}{\lambda_{k}(\vartheta,t)}\,\mathrm{d}t.
\]
Here and in the sequel dot means derivative w.r.t.~$\vartheta $.  The elements
of this matrix have the following expressions
\begin{align*}
\mathrm{I}(\vartheta)_{11} &= \sum_{k=1}^{K}\int_{\tau_{1,k}}
\frac{s_{1,k}'^2(t-\tau_{1,k}) \cos^2(\alpha_{1,k})}{\nu ^2
  \lambda_k(\vartheta,t)}\, \mathrm{d}t,\\*
\mathrm{I}(\vartheta)_{12} &= \sum_{k=1}^{K}\int_{\tau_{1,k}}
\frac{s_{1,k}'^2(t-\tau_{1,k}) \cos(\alpha_{1,k}) \sin(\alpha_{1,k})}{\nu ^2
  \lambda_k(\vartheta,t)}\, \mathrm{d}t,\\
\mathrm{I}(\vartheta)_{13} &= \sum_{k=1}^{K}\int_{\tau_{1,k} \vee \tau_{2,k}}
\frac{s_{1,k}'(t-\tau _{1,k})\, s_{2,k}'(t-\tau_{2,k}) \cos(\alpha _{1,k})
  \cos(\alpha_{2,k})}{\nu ^2 \lambda_k(\vartheta,t)}\, \mathrm{d}t,\\*
\mathrm{I}(\vartheta)_{14} &= \sum_{k=1}^{K}\int_{\tau_{1,k} \vee \tau_{2,k}}
\frac{s_{1,k}'(t-\tau _{1,k})\, s_{2,k}'(t-\tau_{2,k}) \cos(\alpha_{1,k})
  \sin(\alpha_{2,k})}{\nu ^2 \lambda_k(\vartheta,t)}\, \mathrm{d}t.
\end{align*}
The other terms can be written in a similar way.

The regularity conditions are:

\noindent\textit{Conditions\/ $\mathscr{R}$.}
\begin{description}
\itshape
\item[$\mathscr{R}_1$.]For all\/ $i=1,2$ and $k=1,\ldots,K$, the functions\/
  $s_{i,k}(t)=0$ for\/ $t\leq 0$ and\/ $s_{i,k}(t)>0$ for\/ $t>0$.
\item[$\mathscr{R}_2$.]The functions\/ $s_{i,k}(\cdot) \in \mathcal{C}^{2}$,
  $i=1,2$, $k=1,\ldots,K$.  The set\/ $\Theta \subset \mathcal{R}^4$ is open,
  bounded and convex.
\item[$\mathscr{R}_3$.]The Fisher information matrix is uniformly non degenerate
\[
\inf_{\vartheta \in \Theta}\ \inf_{{\bil\|e\bir\|}_4=1} e^\top\,
\mathrm{I}(\vartheta)\, e>0,
\]
where $e\in\mathcal{R}^4$.
\item[$\mathscr{R}_4$.](\textbf{Identifiability}) For any\/ $\varepsilon>0$,
  we have
\[
\inf_{\vartheta_0 \in \Theta}\ \inf_{{\bil\|\vartheta
    -\vartheta_0\bir\|}_4>\varepsilon} \sum_{k=1}^{K}\int_{0}^{T}
\bigl[S_{1,k}(\vartheta^{(1)},t) + S_{2,k}(\vartheta^{(2)},t) -
  S_{1,k}(\vartheta_0^{(1)},t) - S_{2,k}(\vartheta_0^{(2)},t)\bigr]^2\,
\mathrm{d}t>0.
\]
\end{description}

Let us note that in the setup of this section, the identifiability condition
rewrites as follows:
\begin{equation}
\label{r4}
\inf_{\vartheta_0 \in \Theta}\ \inf_{{\bil\|\vartheta
    -\vartheta_0\bir\|}_4>\varepsilon} \sum_{k=1}^{K}\int_{0}^{T}
\bigl[s_{1,k}(t-\tau_{1,k}) + s_{2,k}(t-\tau_{2,k}) -
  s_{1,k}(t-\tau_{1,k}^\circ) - s_{2,k}(t-\tau_{2,k}^\circ)\bigr]^2\,
\mathrm{d}t>0,
\end{equation}
where $\tau_{i,k}^\circ = \tau_{k}(\vartheta_0^{(i)})$, $i=1,2$.

It can be shown that if the conditions $\mathscr{R}_1$--$\mathscr{R}_3$ are
fulfilled, the family of measures $\bigl(\Pb_\vartheta^{(n)},
\vartheta\in\Theta\bigr)$ is locally asymptotically normal (LAN) (see
Lemma~2.1 of~\cite{Kut98}), and therefore we have the Hajek-Le Cam's lower
bound on the risks of an arbitrary estimator $\bar\vartheta_n$ (see,
e.g.,~\cite{IH81})
\begin{equation}
\label{ng}
\lim_{\varepsilon \rightarrow 0}\ \Liminf_{n\rightarrow
  \infty}\ \sup_{{\bil\|\vartheta -\vartheta_0\bir\|}_4 \leq \varepsilon} n\,
\Ex_\vartheta {\bil\|\bar\vartheta_n-\vartheta\bir\|}_4^2 \geq
\Ex_{\vartheta_0} {\bil\|\zeta\bir\|}_4^2 ,\qquad \zeta \sim
\mathcal{N}\bigl(0,\mathrm{I}(\vartheta_0)^{-1}\bigr).
\end{equation}
The asymptotically efficient estimator is defined as an estimator for which
there is equality in the inequality~\eqref{ng} for all $\vartheta_0\in\Theta$.

\begin{theorem}
Let the conditions\/ $\mathscr{R}$ be fulfilled.  Then, uniformly on
compacts\/ $\KK\subset\Theta$, the MLE\/ $\hat\vartheta_n$ and the BEs
$\tilde\vartheta_n$ are consistent and asymptotically normal:
\[
\sqrt{n}\, \bigl(\hat\vartheta_n-\vartheta_0\bigr) \Longrightarrow
\zeta,\qquad\qquad \sqrt{n}\, \bigl(\tilde\vartheta_n-\vartheta_0\bigr)
\Longrightarrow \zeta,
\]
the polynomial moments converge: for any\/ $p>0$, it holds
\[
n^{p/2}\, \Ex_{\vartheta_0}{\bil\|\hat\vartheta_n-\vartheta_0\bir\|}_4^p
\rightarrow \Ex_{\vartheta_0}{\bil\|\zeta\bir\|}_4^p,\qquad n^{p/2}\,
\Ex_{\vartheta_0} {\bil\|\tilde\vartheta_n-\vartheta_0\bir\|}_4^p \rightarrow
\Ex_{\vartheta _0}{\bil\|\zeta\bir\|}_4^p,
\]
and both the MLE and the BEs are asymptotically efficient.
\end{theorem}

\begin{proof}
This theorem is a particular case of Theorems~2.4 and~2.5 of~\cite{Kut98} (see
as well~\cite{Kut82}).  Note that the model of observations with large
intensity asymptotics is equivalent to the model of $n \rightarrow \infty$
independent identically distributed inhomogeneous Poisson processes.  To
verify the condition \textbf{B4} of the~Theorem 2.4, for $\left\|\vartheta
-\vartheta _0\right\|_4\leq \varepsilon $ and sufficiently small $\varepsilon
>0$ we can write
\begin{equation}
\label{05}
\begin{aligned}
&\sum_{k=1}^{K}\int_{0}^{T} \Bigl[\sqrt{\lambda_{k}(\vartheta,t)} -
    \sqrt{\lambda_{k}(\vartheta_0,t)}\Bigr]^2\, \mathrm{d}t = \frac{1}{4}\,
  (\vartheta-\vartheta_0)^\top\, \mathrm{I}(\vartheta_0)\,
  (\vartheta-\vartheta_0)\, \bigl(1+O(\varepsilon)\bigr)\\*
&\qquad\qquad \geq \frac{1}{8}\, (\vartheta-\vartheta_0)^\top\,
  \mathrm{I}(\vartheta_0)\, (\vartheta-\vartheta_0) = \frac{1}{8}\,
         {\bil\|\vartheta-\vartheta_0\bir\|}_4^2\, e^\top\,
         \mathrm{I}(\vartheta_0)\, e \geq \kappa_1
                {\bil\|\vartheta-\vartheta_0\bir\|}_4^2.
\end{aligned}
\end{equation}
Here we used the condition $\mathscr{R}_3$.

For ${\bil\|\vartheta-\vartheta_0\bir\|}_4 \geq \varepsilon$, we have
\begin{align*}
&\sum_{k=1}^{K}\int_{0}^{T} \Bigl[\sqrt{\lambda_{k}(\vartheta,t)} -
    \sqrt{\lambda_{k}(\vartheta_0,t)}\Bigr]^2\, \mathrm{d}t =
  \sum_{k=1}^{K}\int_{0}^{T} \frac{\bigl[{\lambda_{k}(\vartheta,t)} -
      {\lambda_{k}(\vartheta_0,t)}\bigr]^2}{\Bigl[\sqrt{\lambda_{k}(\vartheta,t)}
      + \sqrt{\lambda_{k}(\vartheta_0,t)}\Bigr]^2}\, \mathrm{d}t\\*
&\qquad \quad \geq C\, \sum_{k=1}^{K}\int_{0}^{T}
  \bigl[\lambda_{k}(\vartheta,t) - \lambda_{k}(\vartheta_0,t)\bigr]^2\,
  \mathrm{d}t\\
&\qquad \quad \geq C\, \sum_{k=1}^{K}\int_{0}^{T} \bigl[s_{1,k}(t-\tau_{1,k})
    + s_{2,k}(t-\tau_{2,k}) - s_{1,k}(t-\tau_{1,k}^\circ) -
    s_{1,k}(t-\tau_{1,k}^\circ)\bigr]^2\, \mathrm{d}t\\*
&\qquad \quad \geq C g(\varepsilon).
\end{align*}
Here we used the boundedness of the functions $\lambda_{k}(\vartheta,t)$ and
denoted $g(\varepsilon)>0$ the left hand side of~\eqref{r4}.  Let us denote
$D(\Theta) = \sup_{\vartheta, \tilde\vartheta \in\Theta}
{\bil\|\vartheta-\tilde\vartheta\bir\|}_4$.  Then we have
\begin{equation}
\label{06}
\sum_{k=1}^{K}\int_{0}^{T} \Bigl[\sqrt{\lambda_{k}(\vartheta,t)} -
  \sqrt{\lambda_{k}(\vartheta_0,t)}\Bigr]^2\, \mathrm{d}t \geq C
g(\varepsilon) \geq C g(\varepsilon)\,
\frac{{\bil\|\vartheta-\vartheta_0\bir\|}_4^2}{D(\Theta)^2} \geq \kappa_2
  {\bil\|\vartheta-\vartheta_0\bir\|}_4^2.
\end{equation}
The estimates~\eqref{05} and~\eqref{06} can be joined in
\begin{equation}
\label{07}
\sum_{k=1}^{K}\int_{0}^{T} \Bigl[\sqrt{\lambda_{k}(\vartheta,t)} -
  \sqrt{\lambda_{k}(\vartheta_0,t)}\Bigr]^2\, \mathrm{d}t \geq \kappa
    {\bil\|\vartheta-\vartheta_0\bir\|}_4^2,
\end{equation}
where $\kappa = \kappa_1\wedge\kappa_2$.  Now~\textbf{B4} follows
from~\eqref{07}.
\end{proof}

\subsubsection{Cusp type fronts}

Let us now turn to the intensity function with cusp-type singularity.  Suppose
that the observed Poisson processes $X^K=(X_k(t),\ 0\leq t\leq
T,\ k=1,\ldots,K)$ have intensity functions
\begin{equation}
\label{cusp0}
\lambda_{k,n}(\vartheta,t) = n \psi_{\kappa ,\delta}(t-\tau_{1,k}) S_{1,k}(t)
+ n \psi_{\kappa ,\delta}(t-\tau_{2,k}) S_{2,k}(t) + n\lambda_0 =
n\lambda_k(\vartheta,t),
\end{equation}
where $\kappa \in \bigl(0,\frac{1}{2}\bigr)$ and
\begin{equation}
\label{cusp1}
\psi_{\kappa ,\delta}(t-\tau_{i,k}) =
\Bigl|\frac{t-\tau_{i,k}}{\delta}\Bigr|^\kappa \1_{\{0\leq t-\tau_{i,k}\leq
  \delta\}} + \1_{\{t-\tau_{i,k}>\delta\}},\quad\quad i=1,2,\quad
k=1,\ldots,K.
\end{equation}

Note that the intensity function of the Poisson process recorded by one
detector has two cusp-type singularities.  An example of such an intensity
function is given in Fig.~\ref{F2}.

\begin{figure}[!ht]
\centering
\includegraphics[width=0.7\textwidth]{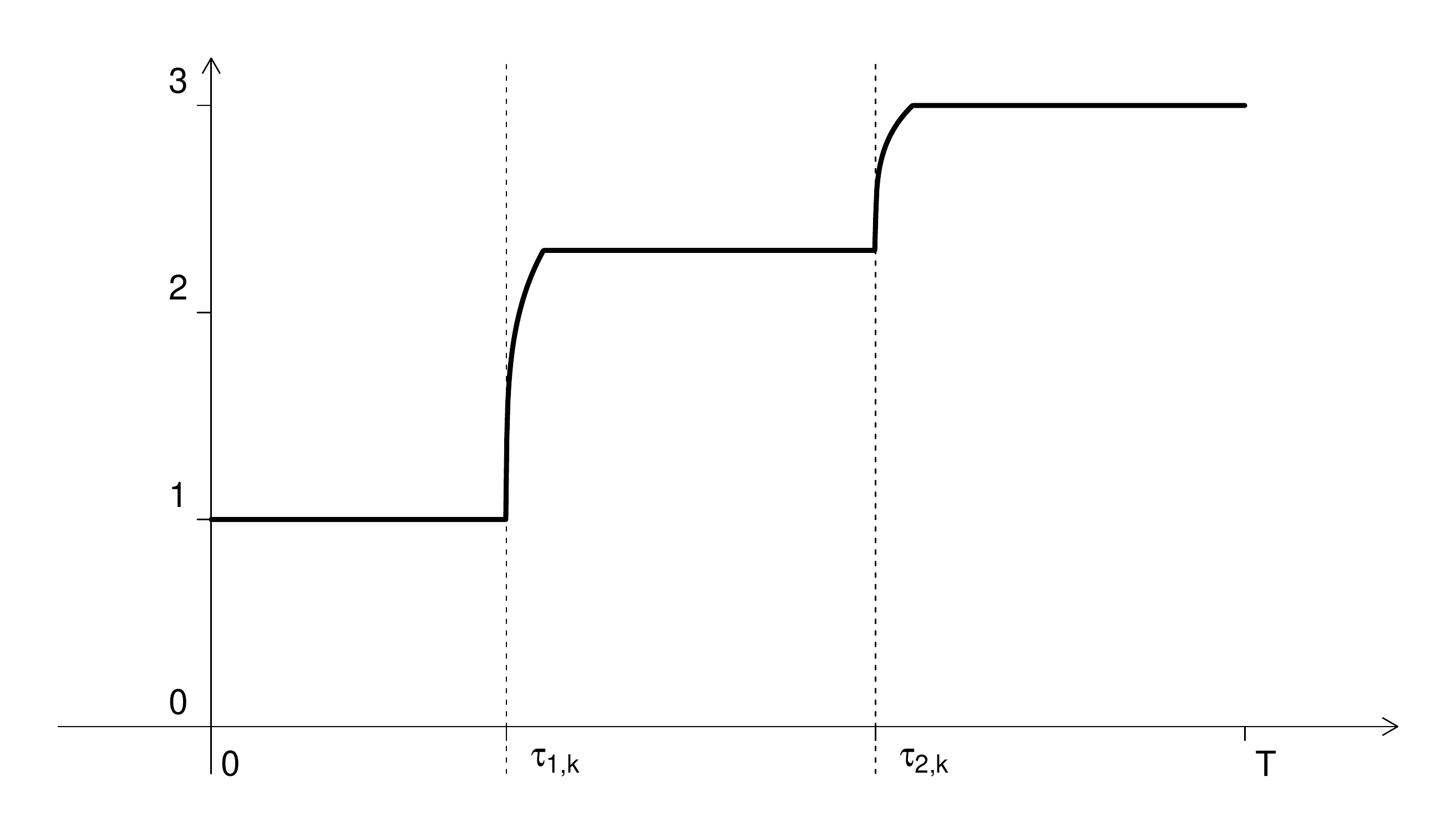}
\caption{Intensity with two cusp type singularities}
\label{F2}
\end{figure}

Recall that
\begin{align*}
&\nu \bigl[\tau_{k}(\vartheta_u^{(1)}) - \tau_{k}(\vartheta_0^{(1)})\bigr]\\*
&\qquad = \bigl[(x_k-x_1'^\circ-u_1\varphi_n)^2 +
    (y_k-y_1'^\circ-u_2\varphi_n)^2\bigr]^{1/2} - \bigl[(x_k-x_1'^\circ)^2 +
    (y_k-y_1'^\circ)^2\bigr]^{1/2}\\
&\qquad = \bigl(\rho_{1,k}-u_1\varphi_n\cos(\alpha_{1,k}^\circ) -
  u_2\varphi_n\sin(\alpha_{1,k}^\circ)\bigr) \bigl(1+O(\varphi_n)\bigr) -
  \rho_{1,k}\\*
 &\qquad = - \langle u^{(1)},m_{1,k}^\circ\rangle\, \varphi_n\,
  \bigl(1+O(\varphi_n)\bigr)
\end{align*}
and $\nu \bigl[\tau_{k}(\vartheta_u^{(2)}) - \tau_{k}(\vartheta_0^{(2)})\bigr]
= - \langle u^{(2)},m_{2,k}^\circ\rangle\, \varphi_n\, \bigl(1+O(\varphi_n)\bigr)$.

Introduce the notations
\begin{align}
&\BB_{i,k} = \bigl(u^{(i)}\,:\, \langle m_{i,k}^\circ,u^{(i)}\rangle < 0\bigr),\quad
\qquad \BB_{i,k}^c=\bigl(u^{(i)}:\; \langle m_{i,k}^\circ,u^{(i)}\rangle \geq
0\bigr),\quad i=1,2,\notag\\*
&I_{i,k}(u^{(i)}) = \hat\Gamma_{i,k}\int_{\mathcal{R}} \Bigl[\bigl|v +
  \nu^{-1} \langle u^{(i)},m_{i,k}^\circ\rangle\bigr|^{\kappa} \1_{\{v\geq
    -\nu^{-1} \langle u^{(i)},m_{i,k}^\circ\rangle\}} - v^{\kappa}\1_{\{v>0\}}
  \Bigr]\, \mathrm{d}W_{i,k}(v),\notag\\
&\Gamma_{1,k}^2 = \frac{S_{1,k}^2(\tau_{i,k}^\circ)}{\delta^{2\kappa}\,
  \lambda_k(\vartheta_0,\tau_{i,k}^\circ)\, \nu^{2\kappa+1}}\,,\qquad
\Gamma_{2,k}^2 = \frac{S_{2,k}^2(\tau_{i,k}^\circ)}{\delta^{2\kappa}\,
  \lambda_k(\vartheta_0,\tau_{i,k}^\circ)\, \nu^{2\kappa+1}}\,,\label{not1}\\
&\hat\Gamma_{1,k} = \frac{S_{1,k}(\tau_{i,k}^\circ)}{\delta^{\kappa}
  \sqrt{\lambda_k(\vartheta_0,\tau_{i,k}^\circ)}}\,, \qquad \qquad
\hat\Gamma_{2,k} =\frac{S_{2,k}(\tau_{i,k}^\circ)}{\delta^{\kappa}
  \sqrt{\lambda_k(\vartheta_0,\tau_{i,k}^\circ)}}\,,\notag\\
&Q_\kappa^2 = \int_{\mathcal{R}} \bigl[\bil|v-1\bir|^{\kappa} \1_{\{v\geq 1\}}
  - v^{\kappa} \1_{\{v>0\}}\bigl]^2\, \mathrm{d}v,\label{not2}\\
&Z_{(k)}(u) = \exp\Biggl(\sum_{i=1}^{2} \biggl[\hat\Gamma _{i,k}\,
  I_{i,k}(u^{(i)}) - \frac{\Gamma_{i,k}^2 Q_\kappa^2 }{2}\, \bigl|\langle
  u^{(i)},m_{i,k}^\circ\rangle\bigr|^{2\kappa+1}\biggl]\Biggr),\notag\\*
&Z(u) = \prod_{k=1}^K Z_{(k)}(u),\qquad \quad\qquad \quad \xi =
\frac{\int_{\mathcal{R}^4} u Z(u)\, \mathrm{d}u}{\int_{\mathcal{R}^4} Z(u)\,
  \mathrm{d}u}\,.\notag
\end{align}
Here $W_{i,k}(v)$, $v\in\mathcal{R}$, $i=1,2$ are two-sided Wiener processes,
i.e., $W_{i,k}(v)=W_{i,k}^+(v)$, $v\geq 0$, and $W_{i,k}(v)=W_{i,k}^-(-v)$,
$v\leq 0$, where $W_{i,k}^\pm(\cdot)$ are independent Wiener processes.

\noindent\textit{Conditions\/ $\mathscr{C}$.}
\begin{description}
\itshape
\item[$\mathscr{C}_1$.]The intensities of the observed processes are given
  by\/~\eqref{cusp0}--\eqref{cusp1}, where the functions\/
  $S_{i,k}(y)\in\mathcal{C}^1$ and are positive.
\item[$\mathscr{C}_2$.]The configuration of the detectors and the set\/
  $\Theta$ are such that all signals from the both sources arrive at the
  detectors during the period\/ $\left[0,T\right]$.
\item[$\mathscr{C}_3$.]The condition $\mathscr{R}_4$ is fulfilled.
\end{description}
We have the following lower bound on the risks of an arbitrary estimator
$\bar\vartheta_n$ of the positions of sources:
\[
\lim_{\varepsilon \rightarrow 0}\ \Liminf_{n \rightarrow
  \infty}\ \sup_{{\bil\|\vartheta - \vartheta_0\bir\|}_4 \leq \varepsilon}
n^{\frac{2}{2\kappa +1}}\, \Ex_\vartheta {\bil\|\bar\vartheta_n -
  \vartheta\bir\|}_4^2 \geq \Ex_{\vartheta_0}{\bil\|\xi\bir\|}_4^2.
\]

This bound is a particular case of a more general bound given in Theorem~1.9.1
of~\cite{IH81}.

\begin{theorem}
\label{T2}
Let the conditions\/ $\mathscr{C}$ be fulfilled.  Then, uniformly on
compacts\/ $\KK\subset\Theta$, the BEs $\tilde\vartheta_n$ are consistent,
converge in distribution:
\[
n^{\frac{1}{2\kappa+1}}\, \bigl(\tilde\vartheta_n - \vartheta_0\bigr)
\Longrightarrow \xi,
\]
the polynomial moments converge: for any\/ $p>0$, it holds
\[
\lim_{n \rightarrow \infty} n^{\frac{p}{2\kappa+1}}\, \Ex_\vartheta
    {\bil\|\tilde\vartheta_n - \vartheta_0\bir\|}_4^p = \Ex_{\vartheta_0}
    {\bil\|\xi\bir\|}_4^p,
\]
and the BEs are asymptotically efficient.
\end{theorem}

\begin{proof}
Let us study the normalized likelihood ratio
\[
Z_n(u) = Z_n^{(\vartheta_0)}(u) =
\frac{L\bigl(\vartheta_0+u\varphi_n,X^K\bigr)}{L\bigl(\vartheta_0,X^K\bigr)},\qquad
u\in\UU_n = \bigl(u\,:\, \vartheta_0+u\varphi_n\in\Theta\bigr),
\]
where $\varphi_n=n^{-\frac{1}{2\kappa+1}}$.  The properties of $Z_n(\cdot)$
which we need are described in the following three lemmas.

\begin{lemma}
\label{L1}
For any compact\/ $\KK\subset\Theta$, the finite-dimensional distributions
of\/ $Z_n(\cdot)$ converge, uniformly on\/ $\vartheta_0\in\Theta$, to the
finite-dimensional distributions of\/ $Z(\cdot)$.
\end{lemma}

\begin{proof}
Put $\vartheta_u = \vartheta_0 + u\varphi_n$, $\vartheta^{(1)}_u =
\vartheta_0^{(1)} + u^{(1)} \varphi_n$ and $\vartheta_u^{(2)} =
\vartheta_0^{(2)} + u^{(2)} \varphi_n$.  For a fixed $u$ denote $\gamma_n =
\sup_{0\leq t\leq T} \lambda _k(\vartheta_0,t)^{-1}
\bigl|\lambda_k(\vartheta_u,t) - \lambda_k(\vartheta_0,t)\bigr|$ and note that
$\gamma_n \rightarrow 0$ as $n \rightarrow \infty$.  Therefore, by Lemma~1.5
of~\cite{Kut22}, the likelihood ratio admits the representation
\begin{align*}
\ln Z_n(u) &= \sum_{k=1}^{K}\int_{0}^{T} \frac{\lambda_k(\vartheta_u,t) -
  \lambda_k(\vartheta_0,t)}{\lambda_k(\vartheta_0,t)}\, \bigl[\mathrm{d}X_k(t)
  - \lambda_k(\vartheta_0,t)\,\mathrm{d}t\bigr] \bigl(1+O(\gamma_n)\bigr)\\*
&\qquad -\frac{n}{2}\, \sum_{k=1}^{K}\int_{0}^{T}
\frac{\bigl[\lambda_k(\vartheta_u,t) -
    \lambda_k(\vartheta_0,t)\bigr]^2}{\lambda_k(\vartheta_0,t)}\, \mathrm{d}t
\bigl(1+O(\gamma_n)\bigr).
\end{align*}
Suppose that $\tau_{1,k}^\circ < \tau_{2,k}^\circ$ and set $2\tau
=\tau_{1,k}^\circ + \tau_{2,k}^\circ$.  Then the second integral can be
written as
\begin{align}
&\int_{0}^{T} \frac{\bigl[\lambda_k(\vartheta_u,t) -
      \lambda_k(\vartheta_0,t)\bigr]^2}{\lambda_k(\vartheta_0,t)}\,
  \mathrm{d}t\notag\\*
&\qquad = \int_{0}^{\tau} \frac{\bigl[\lambda_k(\vartheta_u,t) -
      \lambda_k(\vartheta_0,t)\bigr]^2}{\lambda_k(\vartheta_0,t)}\,
  \mathrm{d}t + \int_{\tau}^{T} \frac{\bigl[\lambda_k(\vartheta_u,t) -
      \lambda_k(\vartheta_0,t)\bigr]^2}{\lambda_k(\vartheta_0,t)}\,
  \mathrm{d}t\notag\\
&\qquad = \int_{0}^{\tau}
  \frac{\Bigl(\bigl[\psi_{\kappa,\delta}\bigl(t-\tau_{k}(\vartheta_u^{(1)})\bigl)
      - \psi_{\kappa,\delta}(t-\tau_{1,k}^\circ)\bigr] S_{1,k}(t) +
    A_k(\vartheta_u,\vartheta_0,t)\Bigr)^2}{\lambda_k(\vartheta_0,t)}\,
  \mathrm{d}t\notag\\*
&\qquad\qquad + \int_{\tau}^{T} \frac{\Bigl(B_k(\vartheta_u,\vartheta_0,t) +
    \bigl[\psi_{\kappa,\delta}\bigl(t-\tau_{k}(\vartheta_u^{(2)})\bigr) -
      \psi_{\kappa,\delta}(t-\tau_{2,k}^\circ)\bigr]
    S_{2,k}(t)\Bigr)^2}{\lambda_k(\vartheta_0,t)}\, \mathrm{d}t\notag\\*
&\qquad = J_{1,k,n}(u^{(1)}) + J_{2,k,n}(u^{(2)})
\label{eq0}
\end{align}
with obvious notations.  Note that for $t \in [0,\tau]$ the function
$A_k(\vartheta,\vartheta_0,t)$ have bounded derivatives w.r.t.\ $\vartheta$,
and that for $t \in [\tau,T]$ the same is true for the function
$B_k(\vartheta,\vartheta_0,t)$.

Suppose that $u^{(1)} \in \BB_{1,k}$.  Then for large $n$ we have
$\tau_{k}(\vartheta_u^{(1)}) > \tau_{1,k}^\circ$, and therefore
\begin{align}
&J_{1,k,n}(u^{(1)}) = n \int_{\tau_{1,k}^\circ}^{\tau}
  \frac{\Bigl(\bigl[\psi_{\kappa,\delta}\bigl(t-\tau_{k}(\vartheta_u^{(1)})\bigr)
      - \psi_{\kappa,\delta}(t-\tau_{1,k}^\circ)\bigr] S_{1,k}(t) +
    A_k(\vartheta_u,\vartheta_0)\Bigr)^2}{\lambda_k(\vartheta_0,t)}\,
  \mathrm{d}t\notag\\*
&\quad = n \int_{\tau_{1,k}^\circ}^{\tau}
  \frac{\Bigl(\bigl[\psi_{\kappa,\delta}\bigl(t-\tau_{k}(\vartheta_u^{(1)})\bigr)
      - \psi_{\kappa,\delta}(t-\tau_{1,k}^\circ)\bigr]
    S_{1,k}(t)\Bigr)^2}{\lambda_k(\vartheta_0,t)}\, \mathrm{d}t + o(1)\notag\\
&\quad = n \int_{0}^{\tau-\tau_{1,k}^\circ}
  \frac{\Bigl(\bigl[\psi_{\kappa,\delta}\bigl(s - \tau_{k}(\vartheta_u^{(1)})
      + \tau_{1,k}^\circ\bigr) - \psi_{\kappa,\delta}(s)\bigr] S_{1,k}(s +
    \tau_{1, k}^\circ)\Bigr)^2}{\lambda_k(\vartheta_0,s+\tau_{1,k}^\circ)}\,
  \mathrm{d}t + o(1)\notag\\
&\quad = n \int_{0}^{\tau-\tau_{1,k}^\circ}
  \frac{\Bigl(\bigl[\psi_{\kappa,\delta}\bigl(s + \nu^{-1}\, \langle
      u^{(1)},m_{1,k}^\circ\rangle\, \varphi_n\bigr) -
      \psi_{\kappa,\delta}(s)\bigr]
    S_{1,k}(s+\tau_{1,k}^\circ)\Bigr)^2}{\lambda_k(\vartheta_0,s+\tau_{1,k}^\circ)}\,
  \mathrm{d}t + o(1)\notag\\
&\quad = n\, \frac{S_{1,k}^2(\tau_{1,k}^\circ)}{\delta^{2\kappa}\,
    \lambda_k(\vartheta_0,\tau_{1,k}^\circ)}\,\biggl[\int_{0}^{-\nu^{-1}\,
      \langle u^{(1)},m_{1,k}^\circ\rangle\, \varphi_n} s^{2\kappa}\,
    \mathrm{d}t\notag\\*
&\qquad\qquad\qquad + \int_{-\nu^{-1}\, \langle u^{(1)},m_{1,k}^\circ\rangle\,
      \varphi_n}^{\tau-\tau_{1,k}^\circ} \Bigl[\bigl|s + \nu^{-1}\, \langle
      u^{(1)},m_{1,k}^\circ\rangle\, \varphi_n\bigr|^{\kappa} -
      s^{\kappa}\Bigr]^2\, \mathrm{d}t\biggr] + o(1)\notag\\
&\quad = n \varphi_n^{2\kappa+1}\, \Gamma_{1,k}^2\, \bigl|\langle
  u^{(1)},m_{1,k}^\circ\rangle\bigr|^{2\kappa+1}\, \biggl[ \int_{0}^{1}
    v^{2\kappa}\, \mathrm{d}v + \int_{1}^{c/\varphi_n}
    \bigl[\bil|v-1\bir|^{\kappa} - v^{\kappa}\bigr]^2\, \mathrm{d}v\biggr] +
  o(1)\notag\\
&\quad = \Gamma_{1,k}^2\, \bigl|\langle
  u^{(1)},m_{1,k}^\circ\rangle\bigr|^{2\kappa+1} \int_{0}^{\infty}
  \bigl[\bil|v-1\bir|^{\kappa}\, \1_{\{v\geq 1\}} - v^{\kappa}\bigr]^2\,
  \mathrm{d}v + o(1)\notag\\*
&\quad = \Gamma_{1,k}^2\, \bigl|\langle
  u^{(1)},m_{1,k}^\circ\rangle\bigr|^{2\kappa+1}\, Q_\kappa^2 + o(1),
\label{eq}
\end{align}
where we changed the variable $s = -\nu^{-1}\, \langle
u^{(1)},m_{1,k}^\circ\rangle\, v$ and used the relation $n
\varphi_n^{2\kappa+1}=1$ and the notations~\eqref{not1} and~\eqref{not2}.

For the values $u^{\left(2\right)}\in \BB_{2,k}$ we have a similar expression
\[
J_{2,k,n}(u^{(2)}) = \Gamma_{2,k}^2\, \bigl|\langle
u^{(2)},m_{2,k}^\circ\rangle\bigr|^{2\kappa+1}\, Q_\kappa^2 + o(1).
\]

If $u^{(1)} \in \BB_{1,k}^c$ and $u^{(2)} \in \BB_{2,k}^c$, then similar
calculations lead to the same integrals.

Hence 
\[
\int_{0}^{T} \frac{\bigl[\lambda_k(\vartheta_u,t) - \lambda
    _k(\vartheta_0,t)\bigr]^2}{\lambda_k(\vartheta_0,t)}\, \mathrm{d}t =
\Bigl[\Gamma_{1,k}^2\, \bigl|\langle
  u^{(1)},m_{1,k}^\circ\rangle\bigr|^{2\kappa+1} + \Gamma_{2,k}^2\,
  \bigl|\langle u^{(2)},m_{2,k}^\circ\rangle\bigr|^{2\kappa+1}\Bigr]\,
Q_\kappa^2 + o(1).
\]

Let us suppose that $\tau_k(\vartheta_0^{(1)}+u^{(1)}\varphi_n) >
\tau_{1,k}^\circ$.  Introduce the centered Poisson process
$\mathrm{d}\pi_{k,n}(t) = \mathrm{d}X_k(t) - \lambda_{k,n}(\vartheta_0,t)\,
\mathrm{d}t$ and consider the stochastic integral
\begin{align*}
&\int_{0}^{T} \frac{\bigl[\lambda_{k,n}(\vartheta_u,t) -
      \lambda_{k,n}(\vartheta_0,t)\bigr]}{\lambda_{k,n}(\vartheta_0,t)}\,
  \mathrm{d}\pi_{k,n}(t)\\*
&\qquad\qquad = \int_{\tau_{1,k}^\circ}^{\tau}
  \frac{\bigl[\lambda_{k}(\vartheta_u,t) -
      \lambda_{k}(\vartheta_0,t)\bigr]}{\lambda_k(\vartheta_0,t)}\,
  \mathrm{d}\pi _{k,n}(t) + \int_{\tau}^{T}
  \frac{\bigl[\lambda_{k}(\vartheta_u,t) -
      \lambda_{k}(\vartheta_0,t)\bigr]}{\lambda_k(\vartheta_0,t)}\,
  \mathrm{d}\pi _{k,n}(t)\\*
&\qquad\qquad = I_{1,k,n}(u^{(1)}) + I_{2,k,n}(u^{(2)}),
\end{align*}
where $2\tau = \tau_{k}(\vartheta_u^{(1)})+\tau_{1,k}^\circ$.

Using the same relations as above, for $u^{(1)}\in\BB_{1,k}$ we can write
\begin{align*}
&I_{1,k,n}(u^{(1)})\\*
&\quad = \int_{0}^{\tau-\tau_{1,k}^\circ}
  \frac{\bigl[\psi_{\kappa,\delta}\bigl(s + \nu^{-1}\, \langle
      u^{(1)},m_{1,k}^\circ\rangle\, \varphi_n\bigr) -
      \psi_{\kappa,\delta}(s)\bigr]
    S_{1,k}(s+\tau_{1,k}^\circ)}{\lambda_k(\vartheta_0,s+\tau_{1,k}^\circ)}\,
  \mathrm{d}\pi_{k,n}(t) + o(1)\\
&\quad = \frac{S_{1,k}(\tau_{1,k}^\circ)}{\delta^{\kappa}\,
    \lambda_k(\vartheta_0,\tau_{1,k}^\circ)}\, \biggl[\int_{0}^{-\nu^{-1}\,
      \langle u^{(1)},m_{1,k}^\circ\rangle\, \varphi_n} s^{\kappa}\,
    \mathrm{d}\pi_{k,n}(s+\tau_{1,k}^\circ)\\*
&\qquad\qquad + \int_{-\nu^{-1}\, \langle u^{(1)},m_{1,k}^\circ\rangle\,
      \varphi_n}^{\tau-\tau_{1,k}^\circ} \Bigl[\bigr|s + \nu^{-1}\, \langle
      u^{(1)},m_{1,k}^\circ\rangle\, \varphi_n\bigr|^{\kappa} -
      s^{\kappa}\Bigr]\, \mathrm{d}\pi_{k,n}(s+\tau_{1,k}^\circ)\biggr] +
  o(1)\\*
&\quad = \hat\Gamma_{1,k} \biggl[\int_{0}^{c/\varphi_n} \Bigl[\bigl|v +
      \nu^{-1}\, \langle u^{(1)},m_{1,k}^\circ\rangle\bigr|^{\kappa}\,
      \1_{\{v\geq -\nu^{-1}\, \langle u^{(1)},m_{1,k}^\circ\rangle \}} -
      v^{\kappa}\Bigr]\, \mathrm{d}W_{1,k,n}(v)\biggr] + o(1).
\end{align*}
Here we changed the variable $s=v\varphi _n$, used the relation
$\sqrt{n}\,\varphi_n^{\kappa +1/2}=1$ and denoted
\[
W_{1,k,n}(v) = \frac{\pi_{k,n}(v \varphi_n+\tau_{1,k}^\circ) -
  \pi_{k,n}(\tau_{1,k}^\circ)}{\sqrt{n \varphi_n \lambda_k
    (\vartheta_0,\tau_{1,k}^\circ)}}\,.
\]

This process has the following first two moments: $\Ex_{\vartheta_0}
W_{1,k,n}(v)=0$,
\begin{align*}
\Ex_{\vartheta_0} W_{1,k,n}(v)^2 & =\frac{1}{\varphi_n\,
  \lambda_k(\vartheta_0,\tau_{1,k}^\circ)}
\int_{\tau_{1,k}^\circ}^{\tau_{1,k}^\circ+v \varphi_n}
\lambda_k(\vartheta_0,t)\, \mathrm{d}t = v \bigl(1+o(1)\bigr) ,\\*
\Ex_{\vartheta_0} W_{1,k,n}(v_1) W_{1,k,n}(v_2) &= \frac{1}{\varphi_n\,
  \lambda_k(\vartheta_0,\tau_{1,k}^\circ)}
\int_{\tau_{1,k}^\circ}^{\tau_{1,k}^\circ+(v_1\wedge v_2)\varphi_n} \lambda
_k(\vartheta_0,t)\, \mathrm{d}t = (v_1\wedge v_2) \bigl(1+o(1)\bigr).
\end{align*}
By the central limit theorem (see, e.g., Theorem~1.2 of~\cite{Kut98})
\begin{align*}
I_{1,k,n}(u^{(1)}) &= \hat\Gamma_{1,k} \int_{0}^{c/\varphi_n}
\Bigl[\bigl|v+\nu^{-1}\, \langle u^{(1)},m_{1,k}^\circ\rangle\bigr|^{\kappa}\,
  \1_{\{v \geq -\nu^{-1}\, \langle u^{(1)},m_{1,k}^\circ\rangle\}} -
  v^{\kappa}\Bigr]\, \mathrm{d}W_{1,k,n}(v)\\*
&\Longrightarrow \hat\Gamma_{1,k} \int_{0}^{\infty} \Bigl[\bigl|v+\nu^{-1}\,
  \langle u^{(1)},m_{1,k}^\circ\rangle\bigr|^{\kappa}\, \1_{\{v\geq -\nu^{-1}
    \langle u^{(1)},m_{1,k}^\circ\rangle\}} - v^{\kappa}\Bigr]\,
\mathrm{d}W_{1,k}^+(v)\\*
&= \hat\Gamma_{1,k}\, I_{1,k}(u^{(1)}),
\end{align*}
where $W_{1,k}^+(v)$, $v \geq 0$, is a standard Wiener process.  A similar
limit for the integral $I_{2,k,n}(u^{(2)})$ can be obtained in the same way.

Suppose that $u^{(2)} \in \BB_{2,k}^c$, i.e., $\langle
u^{(2)},m_{2,k}^\circ\rangle \geq 0$.  Then
\begin{align*}
I_{2,k,n}(u^{(2)}) &= \hat\Gamma_{2,k} \int_{-\nu^{-1}\, \langle
  u^{(2)},m_{2,k}^\circ\rangle}^{c/\varphi_n} \Bigl[\bigl|v+\nu^{-1}\, \langle
  u^{(2)},m_{2,k}^\circ\rangle\bigr|^{\kappa} - v^{\kappa}\, \1_{\{v\geq 0\}}
  \Bigr]\, \mathrm{d}W_{2,k,n}(v)\\*
&\Longrightarrow \hat\Gamma_{2,k} \int_{-\nu^{-1}\, \langle
  u^{(2)},m_{2,k}^\circ\rangle}^{\infty} \Bigl[\bigl|v+\nu^{-1}\, \langle
  u^{(2)},m_{2,k}^\circ\rangle\bigr|^{\kappa} - v^{\kappa}\, \1_{\{v\geq 0\}}
  \Bigr]\, \mathrm{d}W_{2,k}^+(v)\\*
&= \hat\Gamma_{2,k}\, I_{2,k}(u^{(2)}),
\end{align*}

Therefore we obtain
\begin{align*}
\ln Z_{(k),n}(u) &= \int_{0}^{T} \frac{\lambda_k(\vartheta_u,t) -
  \lambda_k(\vartheta_0,t)}{\lambda_k(\vartheta_0,t)}\, \bigl[\mathrm{d}X_k(t)
  - \lambda_k(\vartheta_0,t)\, \mathrm{d}t\bigr]\\*
&\qquad\qquad\qquad\qquad -\frac{n}{2}\, \int_{0}^{T}
\frac{\bigl[\lambda_k(\vartheta_u,t) -
    \lambda_k(\vartheta_0,t)\bigr]^2}{\lambda_k(\vartheta_0,t)}\, \mathrm{d}t
+ o(1)\\
&\Longrightarrow \hat\Gamma_{1,k}\, I_{1,k}(u^{(1)}) - \frac{\Gamma_{1,k}^2
  Q_\kappa^2}{2}\, \bigl|\langle
u^{(1)},m_{1,k}^\circ\rangle\bigr|^{2\kappa+1}\\*
&\qquad\qquad\qquad\qquad + \hat\Gamma_{2,k}\, I_{2,k}(u^{(2)}) -
\frac{\Gamma_{2,k}^2 Q_\kappa^2 }{2}\, \bigl|\langle
u^{(2)},m_{2,k}^\circ\rangle\bigr|^{2\kappa+1}\\*
&=\ln Z_{(k)}(u).
\end{align*}
The Wiener processes $W_{i,k}(\cdot)$, $i=1,2$, $k=1,\ldots,K$, are
independent and this convergence provides the convergence of one-dimensional
distributions
\[
Z_n(u) \Longrightarrow Z(u) = \prod_{k=1}^K Z_{(k)}(u),\qquad \quad u \in
\mathcal{R}^4.
\]

The convergence of finite-dimensional distributions can be proved following
the same lines, but we omit it since it is too cumbersome and does not use new
ideas or tools.
\end{proof}

\begin{lemma}
There exists a constant\/ $C>0$ such that for any\/ $L>0$, it holds
\[
\sup_{\bil|u\bir|<L} \Ex_{\vartheta_0} \bigl|Z_n(u)^{1/2} -
Z_n(u')^{1/2}\bigr|^2 \leq C\, (1+L^{1-2\kappa})\,
{\bil\|u-u'\bir\|}_4^{1+2\kappa}.
\]
\end{lemma}

\begin{proof}
By Lemma~1.5 of~\cite{Kut98}, we have
\begin{equation}
\label{l2}
\begin{aligned}
\Ex_{\vartheta_0} \bigl|Z_n(u)^{1/2} - Z_n(u')^{1/2}\bigr|^2 &\leq
\sum_{k=1}^{K} \int_{0}^{T} \biggl[\sqrt{\lambda_{k,n}(\vartheta_{u},t)} -
  \sqrt{\lambda_{k,n}(\vartheta_{u'},t)}\biggr]^2\, \mathrm{d}t\\*
& \leq C n \sum_{k=1}^{K} \int_{0}^{T} \bigl[\lambda_k(\vartheta_{u},t) -
  \lambda_k(\vartheta_{u'},t)\bigr]^2\, \mathrm{d}t.
\end{aligned}
\end{equation}

Note that the parts of integrals in~\eqref{eq0} containing differentiable
functions $A_k(\vartheta_u,\vartheta_0,t)$ and
$B_k(\vartheta_u,\vartheta_0,t)$ have estimates like
\[
n \int_{0}^{\tau} \bigl[A_k(\vartheta_u,\vartheta_0,t) -
  A_k(\vartheta_{u'},\vartheta_0,t)\bigr]^2\, \mathrm{d}t \leq C\, n
\varphi_n^2\, {\bil\|u-u'\bir\|}^2 = C\, n^{-1+2\kappa}\,
       {\bil\|u-u'\bir\|}_4^2.
\]
Therefore, following~\eqref{eq}, we obtain the relations
\begin{align*}
&\sum_{k=1}^{K} \int_{0}^{T} \bigl[\lambda_k(\vartheta_{u},t) -
    \lambda_k(\vartheta_{u'},t)\bigr]^2\, \mathrm{d}t\\*
&\qquad \qquad \leq C\, \Bigl(\bil{\|u^{(1)}-u'^{(1)}\bir\|}_2^{1+2\kappa} +
  {\bil\|u^{(2)}-u'^{(2)}\bir\|}_2^{1+2\kappa}\Bigr) + C\, n^{-1+2\kappa}\,
  {\bil\|u-u'\bir\|}_4^2\\
&\qquad \qquad \leq C\, {\bil\|u-u'\bir\|}_4^{1+2\kappa} + C\,
  n^{-1+2\kappa}\, {\bil\|u-u'\bir\|}_4^2\\*
&\qquad \qquad \leq C\, (1+L^{1-2\kappa})\, {\bil\|u-u'\bir\|}_4^{1+2\kappa}.
\end{align*}
\end{proof}

\begin{lemma}
\label{L3} 
There exist\/ $c_*>0$ such that
\[
\Pb_{\vartheta_0} \Bigl(Z_T(u) > e^{-c_* {\bil\|u\bir\|}_4^{2\kappa+1}}\Bigr)
\leq e^{-c_* {\bil\|u\bir\|}_4^{2\kappa +1}}
\]
\end{lemma}

\begin{proof} 
By the same Lemma~1.5 of~\cite{Kut98}, we have
\begin{align*}
&\Pb_{\vartheta_0} \Bigl(Z_T(u) > e^{-c_* {\bil\|u\bir\|}_4^{2\kappa+1}}\Bigr)
  \leq e^{\frac{c_*}{2} {\bil\|u\bir\|}_4^{2\kappa +1}}\, \Ex_{\vartheta_0}
  Z_n(u)^{1/2}\\*
&\qquad \qquad = \exp\Biggl(\frac{c_*}{2}\, {\bil\|u\bir\|}_4^{2\kappa+1} -
  \frac{1}{2}\, \sum_{k=1}^{K}\int_{0}^{T}
  \biggl(\sqrt{\lambda_{k,n}(\vartheta_u,t)} -
  \sqrt{\lambda_{k,n}(\vartheta_u,t)}\biggr)^2\,\mathrm{d}t\Biggr).
\end{align*}
Using calculations similar to those of~\eqref{eq0} and~\eqref{eq}, we obtain
for ${\bil\|\vartheta_u-\vartheta_0\bir\|}_4 \leq \varepsilon$ and
sufficiently small $\varepsilon>0$ the estimate
\[
\sum_{k=1}^{K}\int_{0}^{T} \biggl(\sqrt{\lambda_{k,n}(\vartheta_u,t)} -
\sqrt{\lambda_{k,n}(\vartheta_u,t)}\biggr)^2\,\mathrm{d}t \geq c_1\,
     {\bil\|u\bir\|}_4^{2\kappa +1}.
\]

Consider now the case ${\bil\|\vartheta_u-\vartheta_0\bir\|}_4 > \varepsilon$.
Let us denote
\[
g(\varepsilon) = \inf_{{\bil\|\vartheta -
    \vartheta_0\bir\|}_4>\varepsilon}\ \sum_{k=1}^{K}\int_{0}^{T}
\bigl[S_{1,k}(\vartheta^{(1)},t) + S_{2,k}(\vartheta^{(2)},t) -
  S_{1,k}(\vartheta_0^{(1)},t) - S_{2,k}(\vartheta_0^{(2)},t)\bigr]^2\,
\mathrm{d}t
\]
and remark that $\varphi_n\,{\bil\|u\bir\|}_4\leq D =
\smash{\sup\limits_{\vartheta,\vartheta'\in\Theta}}
      {\bil\|\vartheta-\vartheta'\bir\|}_4$.  Hence $n > D^{-(2\kappa+1)}
      {\bil\|u\bir\|}_4^{2\kappa+1}$.  As $g(\varepsilon)>0$ we can write
\begin{align*}
&\sum_{k=1}^{K}\int_{0}^{T} \biggl(\sqrt{\lambda_{k,n}(\vartheta_u,t)} -
\sqrt{\lambda_{k,n}(\vartheta_u,t)}\biggr)^2\,\mathrm{d}t\\*
&\qquad \qquad\qquad \qquad \geq c \sum_{k=1}^{K}\int_{0}^{T}
\bigl[\lambda_{k,n}(\vartheta_u,t) -
  \lambda_{k,n}(\vartheta_0,t)\bigr]^2\,\mathrm{d}t\\*
&\qquad \qquad\qquad \qquad \geq c\,n\,g(\varepsilon) \geq c\,g(\varepsilon)
D^{-(2\kappa+1)}\, {\bil\|u\bir\|}_4^{2\kappa+1} = c_2\,
{\bil\|u\bir\|}_4^{2\kappa+1}.
\end{align*}

Finally, denoting $\bar c = c_1\wedge c_2$ and setting $c_*=\bar c/3$, we
obtain
\[
\frac{c_*}{2}\, {\bil\|u\bir\|}_4^{2\kappa+1} - \frac{1}{2}\,
\sum_{k=1}^{K}\int_{0}^{T} \biggl(\sqrt{\lambda_{k,n}(\vartheta_u,t)} -
\sqrt{\lambda_{k,n}(\vartheta_u,t)}\biggr)^2\,\mathrm{d}t \leq -c_*\,
     {\bil\|u\bir\|}_4^{2\kappa +1}.
\]
\end{proof}

The properties of the process $Z_n(\cdot)$ established in
Lemmas~\ref{L1}-\ref{L3} allow us to cite Theorem~1.10.2 of~\cite{IH81} and,
therefore, to obtain the properties of BEs stated in Theorem~\ref{T2}.
\end{proof}

\subsubsection{Change point type fronts}

Suppose that the intensity functions of the observed inhomogeneous Poisson
processes have jumps at the moments of arrival of the signals, i.e., the
intensities are
\begin{equation}
\label{cp}
\lambda_{k,n}(\vartheta,t) = n \1_{\{t>\tau_{k}(\vartheta^{(1)})\}} S_{1,k}(t)
+ n \1_{\{t>\tau_{k}(\vartheta^{(2)})\}} S_{2,k}(t) + n\lambda_0 = n
\lambda_k(\vartheta,t),
\end{equation}
where $0\leq t\leq T$, $k=1,\ldots,K$.  As before, $\vartheta =
(\vartheta^{(1)},\vartheta^{(2)})^\top \in \Theta$, where $\vartheta^{(1)} =
(\vartheta_1,\vartheta_2)^\top = (x_1',y_1')^\top $ (position of the
source~$\SS_1$) and $\vartheta^{(2)} = (\vartheta_3,\vartheta_4)^\top =
(x_2',y_2')^\top$ (position of the source~$\SS_2$), and
\[
\tau_k(\vartheta^{(i)}) = \nu^{-1} \bigl((x_k-x_i')^2 +
(y_k-y_i')^2\bigr)^{1/2} = \tau_{i,k}.
\]

Introduce the notations
\begin{align*}
\ell_{i,k} &
=\ln\biggl(\frac{\lambda_0}{S_{i,k}(\tau_{i,k}^\circ)+\lambda_0}\biggr),\quad
\BB_{i,k} = \bigl(u^{(i)}\,:\, \langle
u^{(i)},m_{i,k}^\circ\rangle<0\bigr),\ i=1,2,\ k=1,\ldots,K,\\*
Z_{(k)}(u) &= \exp\Biggl(\sum_{i=1}^{2} \Bigl[\ell_{i,k}\,
  x_{i,k}^+\bigl(-\nu^{-1}\, \langle u^{(i)},m_{i,k}^\circ\rangle\bigr) -
  \langle u^{(i)},m_{i,k}^\circ\rangle\, \nu ^{-1}\,
  S_{i,k}(\tau_{i,k}^\circ)\Bigr] \1_{\BB_{i,k}}\\*
&\qquad \quad\qquad \quad -\sum_{i=1}^{2} \Bigl[\ell_{i,k}\,
  x_{i,k}^-\bigl(-\nu^{-1}\, \langle u^{(i)},m_{i,k}^\circ\rangle\bigr) -
  \langle u^{(i)},m_{i,k}^\circ\rangle\, \nu ^{-1}\,
  S_{i,k}(\tau_{i,k}^\circ)\Bigr] \1_{\BB_{i,k}^c}\Biggr),\\*
Z(u) &= \prod_{k=1}^K Z_{(k)}(u),\qquad\quad\qquad\qquad\qquad\quad \eta =
\frac{\int_{\mathcal{R}^4} u\, Z(u)\, \mathrm{d}u}{\int_{\mathcal{R}^4} Z(u)\,
  \mathrm{d}u}\,.
\end{align*}
Here
\[
x_{i,k}^+\bigl(-\nu^{-1}\,\langle u^{(i)},m_{i,k}^\circ\rangle\bigr) =
y_{i,k}^+(s)_{s=-\nu^{-1}\, \langle u^{(i)},m_{i,k}^\circ\rangle},\qquad s\geq
0,
\]
where $y_{i,k}^+(s)$, $s\geq 0$, is a Poisson process with unit intensity.
Similarly
\[
x_{i,k}^-\bigl(-\nu^{-1}\,\langle u^{(i)},m_{i,k}^\circ\rangle\bigr) =
y_{i,k}^-(-s)_{s=\nu ^{-1}\, \langle u^{(i)},m_{i,k}^\circ\rangle},\qquad
s\leq 0,
\]
with another Poisson process $y_{i,k}^-(s)$, $s\geq 0$ with unit intensity.
All Poisson processes $y_{i,k}^\pm(s)$, $s\geq 0$, $i=1,2$, $k=1,\ldots,K$ are
independent.

\noindent\textit{Conditions\/ $\mathscr{D}$.}
\begin{description}
\itshape
\item[$\mathscr{D}_1$.] The intensities of the observed processes
  are~\eqref{cp}, where the functions\/ $S_{i,k}(y) \in \mathcal{C}^1$ and are
  positive.  The set\/ $\Theta \subset \mathcal{R}^4$ is open, bounded and
  convex.
\item[$\mathscr{D}_2$.] The configuration of the detectors and the set\/
  $\Theta$ are such that the signals from the both sources arrive at the
  detectors during the period $[0,T]$.
\item[$\mathscr{D}_3$.] The condition $\mathscr{R}_4$ is fulfilled.
\end{description}

The following lower bound 
\[
\lim_{\varepsilon \rightarrow 0}\ \Liminf_{n \rightarrow
  \infty}\ \sup_{\bil\|\vartheta - \vartheta_0\bir\| \leq \varepsilon} n^2\,
\Ex_{\vartheta}{\bil\|\bar\vartheta_n-\vartheta\bir\|}_4^2 \geq
\Ex_{\vartheta_0}{\bil\|\eta\bir\|}_4^2
\]
holds.  This bound is another particular case of Theorem~1.9.1 of~\cite{IH81}.

\begin{theorem}
\label{T3}
Let the conditions\/ ${\scr D}$ be fulfilled.  Then, uniformly on compacts\/
$\KK\subset\Theta$, the BEs $\tilde\vartheta_n$ are consistent, converge in
distribution:
\[
n\, \bigl(\tilde\vartheta_n-\vartheta_0\bigr) \Longrightarrow \eta,
\]
the polynomial moments converge: for any\/ $p>0$, it holds
\[
n^p\, \Ex_{\vartheta _0}{\bil\|\tilde\vartheta_n-\vartheta_0\bir\|}^p_4
\longrightarrow \Ex_{\vartheta_0}{\bil\|\eta\bir\|}^p_4,
\]
and the BEs are asymptotically efficient.
\end{theorem}

\begin{proof}
The normalized likelihood ratio function in this problem is given by
\begin{align*}
\ln Z_n(u) &= \ln Z_n^{(\vartheta_0)}(u)\\*
&= \sum_{k=1}^{K}\int_{0}^{T} \ln\frac{\lambda_k(\vartheta_0 +
  n^{-1}u,t)}{\lambda_k(\vartheta_0,t)}\, \mathrm{d}X_k(t) -
n\sum_{k=1}^{K}\int_{0}^{T} \bigl[\lambda_k(\vartheta_0+n^{-1}u,t) -
  \lambda_k(\vartheta_0,t)\bigr]\, \mathrm{d}t.
\end{align*}

We have to check once more the conditions of Theorem~1.10.1 of~\cite{IH81} and
to prove three lemmas.

\begin{lemma}
\label{L4}
For any compact\/ $\KK\subset\Theta$, the finite-dimensional distributions
of\/ $Z_n(\cdot)$ converge, uniformly on\/ $\vartheta_0\in\Theta$, to the
finite-dimensional distributions of\/ $Z(\cdot)$.
\end{lemma}

\begin{proof}
Let us set $\vartheta_u = \vartheta_0 + n^{-1} u$ and suppose that $u^{(1)}
\in \BB_{1,k}$ and $u^{(2)} \in \BB_{2,k}$.  For $n$ sufficiently large we
have $\tau_{k}(\vartheta_u^{(1)}) > \tau_{1,k}^\circ$ and
$\tau_{k}(\vartheta_u^{(2)}) > \tau_{2,k}^\circ >
\tau_{k}(\vartheta_u^{(1)})$, and so
\[
\int_{0}^{T} \ln\frac{\lambda_k(\vartheta_u,t)}{\lambda_k(\vartheta_0,t)}\,
  \mathrm{d}X_k(t) = \int_{\tau_{1,k}^\circ}^{\tau_{k}(\vartheta_u^{(1)})}
  \ln\frac{\lambda_0}{S_{1,k}(t)+\lambda_0}\, \mathrm{d}X_k(t) +
  \int_{\tau_{2,k}^\circ}^{\tau_{k}(\vartheta_u^{(2)})}
  \ln\frac{\lambda_0}{S_{2,k}(t)+\lambda_0}\, \mathrm{d}X_k(t).
\]
and
\[
\int_{0}^{T} \bigl[\lambda_k(\vartheta_u,t) -
  \lambda_k(\vartheta_0,t)\bigr]\, \mathrm{d}t =
-\int_{\tau_{1,k}(\vartheta_0)}^{\tau_{1,k}(\vartheta_u)} S_{1,k}(t)\,
\mathrm{d}t - \int_{\tau_{2,k}(\vartheta_0)}^{\tau_{2,k}(\vartheta_u)}
S_{2,k}(t)\, \mathrm{d}t.
\]
For a fixed $u$, using Taylor expansion we can write (recall that
$\tau_{i,k}^\circ = \tau_{k}(\vartheta_0^{(i)})$, $i=1,2$)
\begin{align*}
\int_{\tau_{1,k}^\circ}^{\tau_{k}(\vartheta_u^{(1)})} S_{1,k}(t)\, \mathrm{d}t
&= \bigl(\tau_{k}(\vartheta_u^{(1)}) - \tau_{1,k}^\circ\bigr)\,
S_{1,k}(\tau_{1,k}^\circ) + \frac{1}{n}\,
\int_{\tau_{1,k}^\circ}^{\tau_{k}(\vartheta_u^{(1)})} S_{1,k}'(\tilde t)\,
\mathrm{d}t\\*
&= -(n\nu)^{-1}\, \langle u^{(1)},m_{1,k}^\circ\rangle\,
S_{1,k}(\tau_{1,k}^\circ)\, \bigl(1+O({n}^{-1})\bigr)
\end{align*} 
and
\[
\int_{\tau_{2,k}^\circ}^{\tau_{k}(\vartheta_u^{(2)})} S_{2,k}(t)\, \mathrm{d}t
= -(n\nu)^{-1}\, \langle u^{(2)},m_{2,k}^\circ\rangle\,
S_{2,k}(\tau_{2,k}^\circ)\, \bigl(1+O({n}^{-1})\bigr).
\]

For stochastic integrals, similar calculations provide
\[
\int_{\tau_{1,k}^\circ}^{\tau_{k}(\vartheta_u^{(1)})}
\ln\frac{\lambda_0}{S_{1,k}(t)+\lambda_0}\, \mathrm{d}X_k(t) =
\ln\frac{\lambda_0}{S_{1,k}(\tau_{1,k}^\circ)+\lambda_0}\,
\bigl[X_k\bigl(\tau_{k}(\vartheta_u^{(1)})\bigr) - X_k(\tau_{1,k}^\circ)\bigr]
\bigl(1+O(n^{-1})\bigr).
\]
Further
\begin{align*}
X_k\bigl(\tau_{k}(\vartheta_u^{(1)})\bigr) - X_k(\tau_{1,k}^\circ) &=
X_k\bigl(\tau_{1,k}^\circ - n^{-1} \nu ^{-1} \langle
u^{(1)},m_{1,k}^\circ\rangle\bigr) - X_k(\tau_{1,k}^\circ)\\*
&\qquad + X_k\bigl(\tau_{k}(\vartheta_u^{(1)})\bigr) -
X_k\bigl(\tau_{1,k}^\circ - n^{-1} \nu ^{-1} \langle
u^{(1)},m_{1,k}^\circ\rangle\bigr)\\*
&= x_{1,k}^+\bigl(-\nu^{-1} \langle u^{(1)},m_{1,k}^\circ\rangle\bigr) +
O(n^{-1/2}).
\end{align*}
Here $x_{1,k}^+\bigl(-\nu^{-1}\, \langle u^{(1)},m_{1,k}^\circ\rangle\,\bigr)
= X_k\bigl(\tau_{1,k}^\circ - n^{-1} \nu^{-1} \langle
u^{(1)},m_{1,k}^\circ\rangle\bigr) - X_k(\tau_{1,k}^\circ)$ is a Poisson
random process and the last estimate was obtained as follows:
\begin{align*}
&\Ex_{\vartheta _0}\bigl[X_k\bigl(\tau_{k}(\vartheta_u^{(1)})\bigr) -
    X_k\bigl(\tau_{1,k}^\circ - n^{-1} \nu ^{-1} \langle
    u^{(1)},m_{1,k}^\circ\rangle\bigr)\bigr]^2\\*
&\qquad = n \int_{\tau_{k}(\vartheta_u^{(1)})}^{\tau_{1,k}^\circ - n^{-1}
    \nu^{-1} \langle u^{(1)},m_{1,k}^\circ\rangle} \lambda_k(\vartheta_0,t)\,
  \mathrm{d}t + \biggl(n \int_{\tau_{k}(\vartheta_u^{(1)})}^{\tau_{1,k}^\circ
    - n^{-1} \nu^{-1} \langle u^{(1)},m_{1,k}^\circ\rangle}
  \lambda_k(\vartheta_0,t)\, \mathrm{d}t\biggr)^2\\*
&\qquad \leq \frac{C}{n}\,.
\end{align*}
Hence, if $u^{(1)} \in \BB_{1,k}$ and $u^{(2)} \in \BB_{2,k}$, it holds
\begin{align*}
&\int_{0}^{T} \ln\frac{\lambda_k(\vartheta_0 + n^{-1}
    u,t)}{\lambda_k(\vartheta_0,t)}\, \mathrm{d}X_k(t) =
  \ln\biggl(\frac{\lambda_0}{S_{1,k}(\tau_{1,k}^\circ) + \lambda_0}\biggr)\,
  x_{1,k}^+\bigl(-\nu^{-1} \langle u^{(1)},m_{1,k}^\circ\rangle\bigr)\\*
&\qquad\qquad
  +\ln\biggl(\frac{\lambda_0}{S_{2,k}(\tau_{2,k}^\circ)+\lambda_0}\biggr)\,
  x_{2,k}^+\bigl(-\nu^{-1} \langle u^{(2)},m_{2,k}^\circ\rangle\bigr) +
  O(n^{-1/2}).
\end{align*}
If $u^{(1)} \in \BB_{1,k}^c$ and $u^{(2)} \in \BB_{2,k}^c$, we have
\begin{align*}
&\int_{0}^{T} \ln\frac{\lambda _k(\vartheta_0 + n^{-1}
    u,t)}{\lambda_k(\vartheta_0,t)}\, \mathrm{d}X_k(t) = \ln\biggl(1 +
  \frac{S_{1,k}(\tau_{1,k}^\circ)}{\lambda_0}\biggr)\,
  x_{1,k}^-\bigl(-\nu^{-1} \langle u^{(1)},m_{1,k}^\circ\rangle\bigr)\\*
&\qquad\qquad + \ln\biggl(1 +
  \frac{S_{2,k}(\tau_{2,k}^\circ)}{\lambda_0}\biggr)\, x_{2,k}^-\bigl(\nu
  ^{-1} \langle u^{(2)},m_{2,k}^\circ\rangle\bigr) + O(n^{-1/2}).
\end{align*}
\end{proof}

\begin{lemma}
There exists a constant\/ $C>0$ such that
\[
\Ex_{\vartheta_0}\bigl|Z_n(u)^{1/2} - Z_n(u')^{1/2}\bigr|^2 \leq C\,
   {\bil\|u-u'\bir\|}_4.
\]
\end{lemma}

\begin{proof}
We have (see~\eqref{l2})
\[
\sup_{\vartheta_0\in\KK} \Ex_{\vartheta_0}\bigl|Z_n(u)^{1/2} -
Z_n(u')^{1/2}\bigr|^2 \leq \sum_{k=1}^{K} n \int_{0}^{T}
\bigl[\lambda_k(\vartheta_{u},t) - \lambda_k(\vartheta_{u'},t)\bigr]^2\,
\mathrm{d}t.
\]

Suppose that $\tau_{k}(\vartheta_u^{(1)}) > \tau_{k}(\vartheta_{u'}^{(1)})$
and $\tau_{k}(\vartheta_u^{(2)}) > \tau_{k}(\vartheta_{u'}^{(2)}) >
\tau_{k}(\vartheta_u^{(1)})$.  Then
\begin{align*}
&n \int_{0}^{T} \Bigl[S_{1,k}(t)\, \1_{\{\tau_{k}(\vartheta_{u'}^{(1)}) \leq t
      \leq \tau_{k}(\vartheta_u^{(1)})\}} + S_{2,k}(t)\,
    \1_{\{\tau_{k}(\vartheta_{u'}^{(2)}) \leq t \leq
      \tau_{k}(\vartheta_u^{(2)})\}}\Bigr]^2\, \mathrm{d}t\\*
&\qquad\qquad = n
  \int_{\tau_{k}(\vartheta_{u'}^{(1)})}^{\tau_{k}(\vartheta_u^{(1)})}
  S_{1,k}^2(t)\, \mathrm{d}t + n
  \int_{\tau_{k}(\vartheta_{u'}^{(2)})}^{\tau_{k}(\vartheta_u^{(2)})}
  S_{2,k}^2(t)\, \mathrm{d}t\\
&\qquad\qquad \leq C n \bigl(\tau_{k}(\vartheta_u^{(1)}) -
  \tau_{k}(\vartheta_{u'}^{(1)})\bigr) + C n \bigl(\tau_{k}(\vartheta_u^{(2)})
  - \tau_{k}(\vartheta_{u'}^{(2)})\bigr)\\
&\qquad\qquad \leq C \Bigl(\bigl|\langle u'^{(1)} -
  u^{(1)},m_{1,k}^\circ\rangle\bigr| + \bigl|\langle u'^{(2)} -
  u^{(2)},m_{2,k}^\circ\rangle\bigr|\Bigr)\\*
&\qquad\qquad \leq C\bigl({\bil\|u'^{(1)} - u^{(1)}\bir\|}_2 + {\bil\|u'^{(2)}
    - u^{(2)}\bir\|}_2\bigr) \leq C\, {\bil\|u'-u\bir\|}_4.
\end{align*}
\end{proof}

\begin{lemma}
\label{L6} 
There exists a constant\/ $\kappa >0$ such that
\[
\Ex_{\vartheta _0} Z_n(u)^{1/2} \leq e^{-\kappa\, {\bil\|u\bir\|}_4}.
\] 
\end{lemma}

\begin{proof}
We have
\begin{align*}
\ln\Ex_{\vartheta_0}Z_n(u)^{1/2} &= -\frac{1}{2}\, \sum_{k=1}^{K}\int_{0}^{T}
\biggl[\sqrt{\lambda_{k,n}(\vartheta_u,t)} -
  \sqrt{\lambda_{k,n}(\vartheta_0,t)}\biggr]^2\, \mathrm{d}t\\*
&\leq -C \sum_{k=1}^{K}\int_{0}^{T} \bigl[\lambda_{k,n}(\vartheta_u,t) -
  \lambda_{k,n}(\vartheta_0,t)\bigr]^2\, \mathrm{d}t.
\end{align*}
Suppose that $\tau_{k}(\vartheta_u^{(i)}) > \tau_{i,k}^\circ$,
$\tau_{2,k}^\circ > \tau_{k}(\vartheta_u^{(1)}) $ and ${\bil\|\vartheta_u -
  \vartheta_0\bigr\|}_4 \leq \varepsilon$.  Then
\begin{align*}
&\int_{0}^{T} \bigl[\lambda_{k,n}(\vartheta_u,t) -
    \lambda_{k,n}(\vartheta_0,t)\bigr]^2\, \mathrm{d}t\\*
&\qquad\qquad \geq c\, n \int_{0}^{T} \Bigl[S_{1,k}(t)\,
    \1_{\{\tau_{1,k}^\circ < t < \tau_{k}(\vartheta_u^{(1)})\}} + S_{2,k}(t)\,
    \1_{\{\tau_{2,k}^\circ < t < \tau_{k}(\vartheta_u^{(2)})\}}\Bigr]^2\,
  \mathrm{d}t\\
&\qquad\qquad \geq c\, n
  \biggl[\int_{\tau_{1,k}^\circ}^{\tau_{k}(\vartheta_u^{(1)})} S^2_{1,k}(t)\,
    \mathrm{d}t + \int_{\tau_{2,k}^\circ}^{\tau_{k}(\vartheta_u^{(2)})}
    S^2_{2,k}(t)\, \mathrm{d}t\biggr]\\
&\qquad\qquad \geq c\, n \bigl[\tau_{k}(\vartheta_u^{(1)}) - \tau_{1,k}^\circ
    + \tau_{k}(\vartheta_u^{(2)}) - \tau_{2,k}^\circ\bigr]\\
&\qquad\qquad \geq c \bigl[-\langle u^{(1)},m_{1,k}^\circ\rangle - \langle
    u^{(2)},m_{2,k}^\circ\rangle\bigr]\\*
&\qquad\qquad \geq c\, \frac{-\langle u^{(1)},m_{1,k}^\circ\rangle - \langle
    u^{(2)},m_{2,k}^\circ\rangle}{{\bil\|u\bir\|}_4}\, {\bil\|u\bir\|}_4 =
  -c\, \sqrt{2}\, \langle \mathrm{e},\mathrm{m}_k\rangle\, {\bil\|u\bir\|}_4.
\end{align*}
Here $-\langle \mathrm{e},\mathrm{m}_k\rangle > 0$ and the vectors are
\[
\mathrm{e} = \frac{u}{{\bil\|u\bir\|}_4},\quad {\bil\|\mathrm{e}\bir\|}_4 =
1,\quad \mathrm{m}_k = \biggl(\frac{x_k-x_1'^\circ}{\sqrt{2}\rho_k}\, ,
\frac{y_k-y_1'^\circ}{\sqrt{2}\rho_k}\, ,
\frac{x_k-x_2'^\circ}{\sqrt{2}\rho_k}\, ,
\frac{y_k-y_2'^\circ}{\sqrt{2}\rho_k}\biggr), \quad
     {\bil\|\mathrm{m}_k\bir\|}_4 = 1.
\]
Note that
\[
\inf_{{\bil\|\mathrm{e}\bir\|}_4=1} \sum_{k=1}^{K} \bigl(-\langle
\mathrm{e},\mathrm{m}_k\rangle\bigr) = \bar c > 0.
\]
Indeed, if for some $\mathrm{e}$ we have $\bar c=0$, then all scalar products
$\langle u^{(i)},m_{i,k}^\circ\rangle = 0$, $i=1,2$, $k=1,\ldots,K$, but such
configuration of detectors is impossible.

Let ${\bil\|\vartheta_u - \vartheta_0\bir\|}_4 > \varepsilon$.  Then as in the
proof of Lemma~\ref{L3} we obtain the estimate
\[
\sum_{k=1}^{K}\int_{0}^{T} \bigl[\lambda_{k,n}(\vartheta_u,t) -
  \lambda_{k,n}(\vartheta_0,t)\bigr]^2\, \mathrm{d}t \geq n\, g(\varepsilon)
\geq D^{-1}\, {\bil\|u\bir\|}_4 = \tilde c\, {\bil\|u\bir\|}_4.
\]
Therefore
\[
\sum_{k=1}^{K}\int_{0}^{T} \biggl[\sqrt{\lambda_{k,n}(\vartheta_u,t)} -
  \sqrt{\lambda_{k,n}(\vartheta_0,t)}\biggr]^2\, \mathrm{d}t \geq 2\kappa\,
    {\bil\|u\bir\|}_4,
\]
with a corresponding $\kappa>0$.
\end{proof}

Once more, according to Theorem~1.10.2 of~\cite{IH81}, the properties of the
normalized likelihood ratio $Z_n(\cdot)$ established in Lemmas
\ref{L4}--\ref{L6} provide us the properties of the BEs stated in
Theorem~\ref{T3}.
\end{proof}

\section{On identifiability}

Recall that we consider the situation where $K\in\mathbb{N}^*$ detectors
$\DD_1,\ldots,\DD_K\in\mathbb{R}^2$ receive signals from two sources
$\SS_1,\SS_2\in\mathbb{R}^2$.  In this section we make no distinction between
detectors/sources and their positions, and we suppose that the two signals
received by a detector are identical, i.e., we have $S_{1,k}(\cdot) =
S_{2,k}(\cdot) = S_k(\cdot)$ in the right hand side of~\eqref{signals}, and so
the signals recorded by the $k$-th detector are
\[
S_{i,k}(\vartheta^{(i)},t) = \psi(t-\tau_{i,k})\, S_k(t),\qquad i=1,2.
\]

Considering the identifiability condition $\mathscr{R}_4$, we can see that if
for some $\vartheta \in\Theta $ and some $\varepsilon>0$ satisfying
${\bil\|\vartheta - \vartheta_0\bir\|}_4 \geq \varepsilon >0$ we have
\[
\sum_{k=1}^{K}\int_{0}^{T} \bigl[\psi(t-\tau_{1,k})\, S_k(t) +
  \psi(t-\tau_{2,k})\, S_k(t) - \psi(t-\tau_{1,k}^\circ)\, S_k(t) -
  \psi(t-\tau _{2,k}^\circ)\, S_k(t)\bigr]^2\, \mathrm{d}t = 0,
\]
then, taking into account that $S_k(t)>0$ and that $\psi(t)=0$ for $t<0$ and
$\psi(t)\ne 0$ for~$t>0$, we should have (for each $k$) either
\[
\tau_{1,k}=\tau_{1,k}^\circ \quad\text{and}\quad \tau_{2,k}=\tau_{2,k}^\circ
\]
or
\[
\tau_{1,k}=\tau_{2,k}^\circ \quad\text{and}\quad \tau_{2,k}=\tau_{1,k}^\circ.
\]
Of course, in such situation the consistent estimation is impossible, because
for two different values of the unknown parameter we have the same statistical
model.

We see that the question of identifiability is reduced to the following one:
\textit{when having the distances from each detector to the two sources
  (without knowing which distance corresponds to which source) is it possible
  to find\/ $\vartheta$?} So, the system will be identifiable if and only if
there exist no two different pairs of sources providing the same $K$ pairs of
distances to the detectors.

We introduce the following definition.

\begin{definition}
We say that\/ $n\in\mathbb{N}^*$ points\/ $A_1,\ldots,A_n\in\mathbb{R}^2$
``lay on a cross'' if there exist a pair of orthogonal lines\/~$\ell_1$
and\/~$\ell_2$ such that\/ $A_i\in \ell_1\cup\ell_2$ for all\/~$i=1,\ldots,n$.
\end{definition}

Now we can state the theorem providing a necessary and sufficient condition
for identifiability.

\begin{theorem}
A sufficient condition for a system with\/ $K$ detectors to be identifiable is
that the detectors do not lay on a cross.

In absence of restrictions on the positions of the sources (i.e., if we could
suppose that\/~$\Theta = \mathcal{R}^4$) this condition would be also
necessary.
\end{theorem}

Before proving this theorem, let us note that at least $4$ detectors are
necessary for the system to be identifiable.  Indeed, any three detectors
$\DD_1,\DD_2,\DD_3$ lay on a cross (take, for example, the line
$\overleftrightarrow{\,\DD_1 \DD_2}$ and the perpendicular to this line passing
by~$\DD_3$).

Note also that $4$ detectors do not lay on a cross if (and only if) they are
in general linear position (any $3$ of them are not aligned) and cannot be
split in two pairs such that the lines passing by these pairs are orthogonal.

Finally, note that if $5$ (or more) detectors are in general linear position,
then they necessarily do not lay on a cross (and hence the system is
identifiable).  Indeed, if they laid on a cross, at least one of the lines
forming the cross would contain at least three of them, and so, they would not
be in general linear position.

\begin{proof}
Note that if the system is identifiable with $\Theta = \mathcal{R}^4$, it will
be also identifiable with any $\Theta \subset \mathcal{R}^4$.  So we can
suppose $\Theta = \mathcal{R}^4$ and reformulate the theorem using the
contrapositions as follows: \textit{the detectors $\DD_1,\ldots,\DD_K$ lay on
  a cross if and only if there exist two different pairs of sources providing
  the same\/ $K$ pairs of distances to them}.

In order to show the necessity, we suppose that the detectors $\DD_1,\ldots,\DD_K$
lay on a cross formed by a pair of orthogonal line~$\ell_1$ and~$\ell_2$, and
we need to find two different pairs of sources~$\{\SS_1,\SS_2\}$
and~$\{\SS_1',\SS_2'\}$ providing the same $K$ pairs of distances to the
detectors.

Let us denote $O$ the point of intersection of~$\ell_1$ and~$\ell_2$, take an
arbitrary point $\SS_1$ not belonging to neither~$\ell_1$, nor~$\ell_2$, and
denote~$\SS_1'$, $\SS_2'$ and~$\SS_2$ the points symmetric to~$\SS_1$ with respect
to~$\ell_1$, $\ell_2$ and~$O$ respectively (see Fig.~\ref{cross2}).

\begin{figure}[!ht]
\centering
\includegraphics{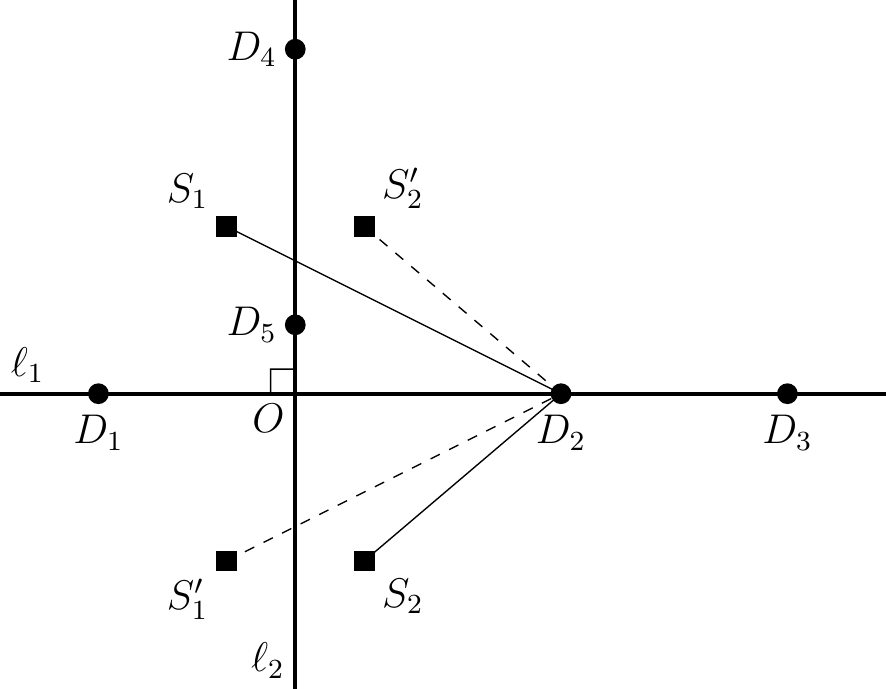}
\caption{Non identifiability for detectors laying on a cross}
\label{cross2}
\end{figure}

Then, clearly, the pairs of sources~$\{\SS_1,\SS_2\}$ and~$\{\SS_1',\SS_2'\}$ provide
the same pair of distances to any point of a cross, and hence the same $K$
pairs of distances to the detectors~$\DD_1,\ldots,\DD_K$.

Now we turn to the proof of the sufficiency.  We suppose that two different
pairs of sources~$\{\SS_1,\SS_2\}$ and~$\{\SS_1',\SS_2'\}$ provide the same $K$ pairs
of distances to the detectors~$\DD_1,\ldots,\DD_K$, and we need to show that the
detectors lay on a cross.

We need to distinguish several cases.

If in one of the two pairs the sources are located at the same point (say
$\SS_1=\SS_2$), then at least one of the points~$\SS_1'$ and~$\SS_2'$ (say $\SS_1'$)
must be different from~$\SS_1$ and~$\SS_2$ (see the left picture in
Fig.~\ref{partcases}).  Then, if our two pairs of sources provide the same
pair of distances to a point $D$, we should have, in particular,
$\rho(D,\SS_1)=\rho(D,\SS_1')$, and hence $D\in b_{\SS_1 \SS_1'}$.  Here and in the
sequel, we denote~$\rho$ the Euclidean distance, and for any two distinct
points~$A$ and~$B$, we denote~$b_{AB}$ the perpendicular bisector of the
segment~$\overline{A B}$.  Therefore, all the detectors must belong to the
line~$b_{\SS_1 \SS_1'}$ (and, in particular, they lay on a cross).

\begin{figure}[!ht]
\centering
\includegraphics{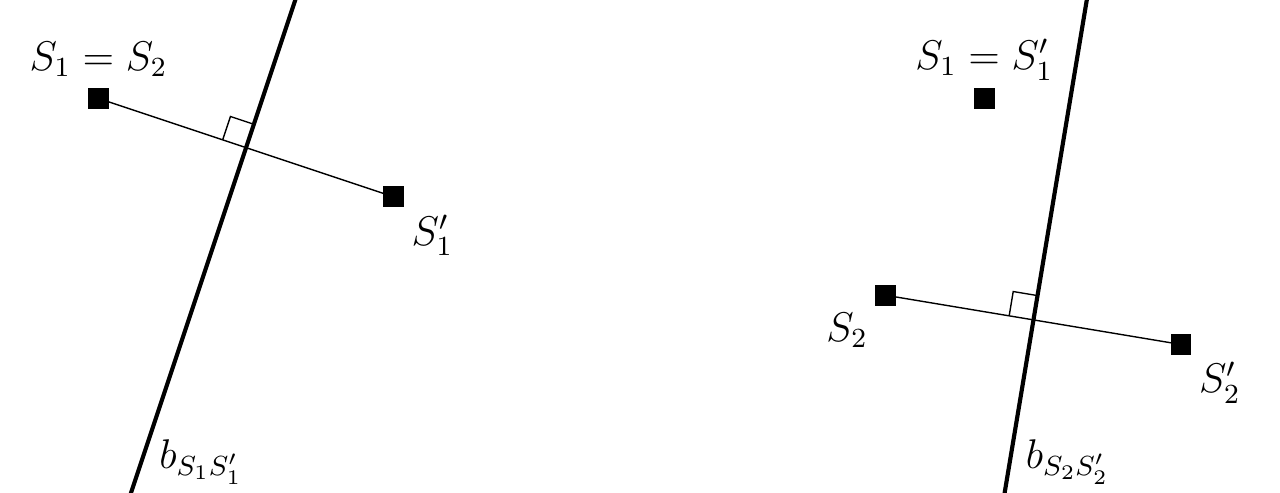}
\caption{Cases with coinciding sources}
\label{partcases}
\end{figure}

So, from now on, we can suppose that $\SS_1\ne \SS_2$ and $\SS_1'\ne \SS_2'$.

Now we consider the case when the two pairs of sources have a point in common
(say~$\SS_1=\SS_1'$).  Note that in this case, we must have $\SS_2\ne \SS_2'$,
since otherwise the pairs of sources will not be different (see the left
picture in Fig.~\ref{partcases}).  Then, our two pairs of sources provide the
same pair of distances to a point $D$ if and only if
$\rho(D,\SS_2)=\rho(D,\SS_2')$, which is equivalent to $D\in b_{\SS_2
  \SS_2'}$.  Therefore, all the detectors must belong to the line~$b_{\SS_2
  \SS_2'}$ (and, in particular, they lay on a cross).

So, from now on, we can suppose that the points $\SS_1,\SS_2,\SS_1',\SS_2'$ are all
different.  In this case, our two pairs of sources provide the same pair of
distances to a point $D$ if and only if $D$ belongs at the same time either to
the pair of lines~$b_{\SS_1 \SS_1'}$ and~$b_{\SS_2 \SS_2'}$, or to the pair of
lines~$b_{\SS_1 \SS_2'}$ and~$b_{\SS_2 \SS_1'}$ $\bigl($otherwise speaking, if and
only if $D\in (b_{\SS_1 \SS_1'}\cap b_{\SS_2 \SS_2'}) \cup (b_{\SS_1 \SS_2'}\cap b_{\SS_2
  \SS_1'})\bigr)$.  We need again to distinguish several (sub)cases depending on
whether the lines in each of these pairs coincide or not.

First we suppose that~$b_{\SS_1 \SS_1'}\ne b_{\SS_2 \SS_2'}$ and~$b_{\SS_1 \SS_2'}\ne
b_{\SS_2 \SS_1'}$ (see the left picture in Fig.~\ref{gencases}).  Then each of
these pairs of lines meet in at most one point.  So, all the detectors must
belong to a set consisting of at most two points (and, therefore, there is at
most two detectors and they trivially lay on a cross).

\begin{figure}[!ht]
\centering
\includegraphics{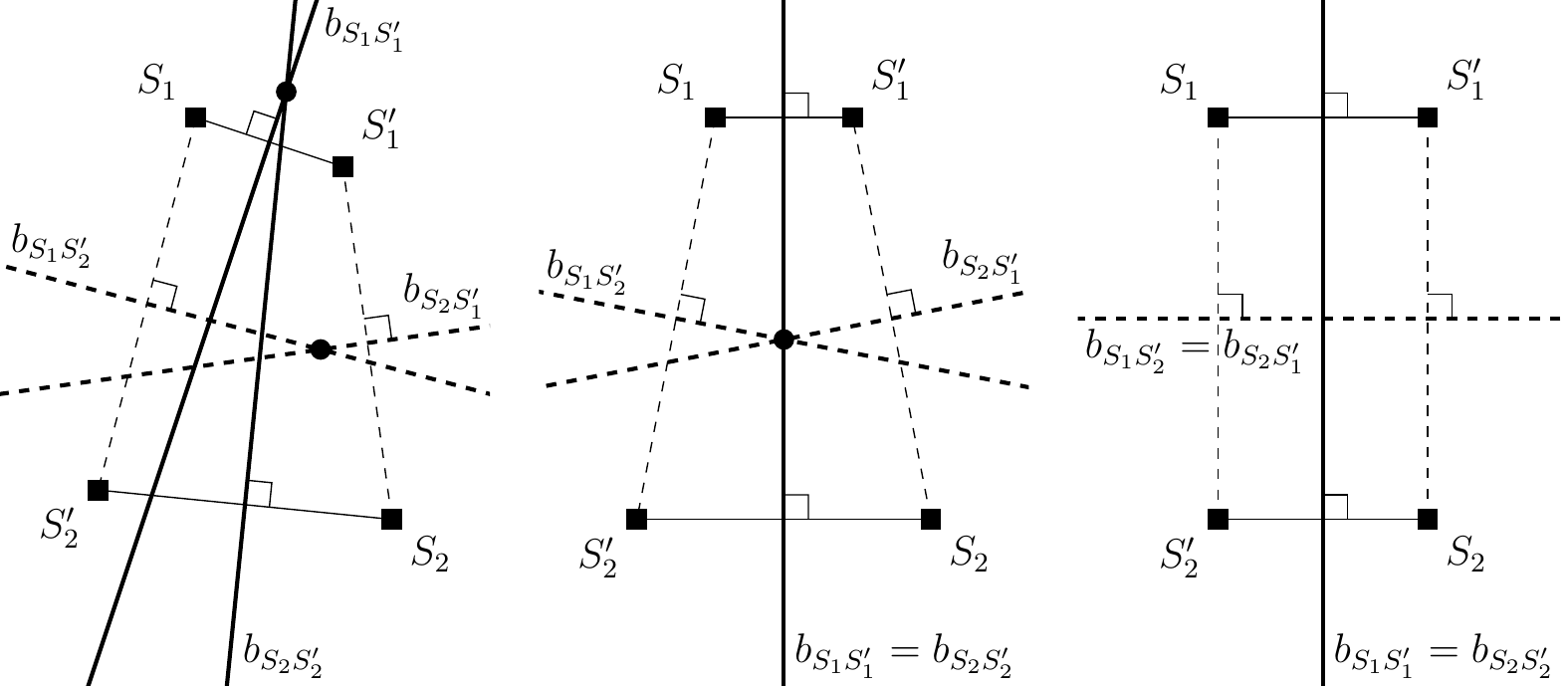}
\caption{Cases without coinciding sources}
\label{gencases}
\end{figure}

Now we suppose that the lines coincide in one of the pairs, and not in the
other (say~$b_{\SS_1 \SS_1'}=b_{\SS_2 \SS_2'}$ and~$b_{\SS_1 \SS_2'}\ne b_{\SS_2 \SS_1'}$).
Note that in this case, the points~$\SS_1$ and~$\SS_1'$, as well as the
points~$\SS_2$ and~$\SS_2'$, are symmetric with respect to the line~$b_{\SS_1
  \SS_1'}$, and hence the lines~$b_{\SS_1 \SS_2'}$ and~$b_{\SS_2 \SS_1'}$ either do not
meet, or meet in a point belonging to~$b_{\SS_1 \SS_1'}$ (see the middle picture
in Fig.~\ref{gencases}).  So, finally all the detectors must belong to the
line~$b_{\SS_1 \SS_1'}$ (and, in particular, they lay on a cross).

It remains to consider the case when~$b_{\SS_1 \SS_1'}=b_{\SS_2 \SS_2'}$ and~$b_{\SS_1
  \SS_2'}=b_{\SS_2 \SS_1'}$ (see the right picture in Fig.~\ref{gencases}).  In
this case, as the segments~$\overline{\SS_1 \SS_1'}$ and~$\overline{\SS_2 \SS_2'}$
have a common perpendicular bisector, they must, in particular, be parallel
(including the case where they lay on a same line).  The same reasoning
applies to the segments~$\overline{\SS_1 \SS_2'}$ and~$\overline{\SS_2 \SS_1'}$.  Now,
the case when the points $\SS_1,\SS_1',\SS_2,\SS_2'$ lay on a same line is impossible,
since these points being all different, we clearly can not have common
perpendicular bisectors at the same time for~$\overline{\SS_1 \SS_1'}$
and~$\overline{\SS_2 \SS_2'}$ and for~$\overline{\SS_1 \SS_2'}$ and~$\overline{\SS_2
  \SS_1}$.  Thus, $\SS_1 \SS_1' \SS_2 \SS_2'$ is a parallelogram.  Moreover, as its
opposite sides have common perpendicular bisectors, $\SS_1 \SS_1' \SS_2 \SS_2'$ is
necessarily a rectangle.  But in this case, the perpendicular
bisectors~$b_{\SS_1 \SS_1'}$ and~$b_{\SS_1 \SS_2'}$ are orthogonal.  Hence, as all the
detectors belong to~$b_{\SS_1 \SS_1'}\cup b_{\SS_1 \SS_2'}$, they lay again on a
cross.
\end{proof}

\section{Discussions}

The considered models in the cases of cusp and change-point singularities can
be easily generalized to the models with signals
$s_{i,k}\bigl(t-\tau_k(\vartheta)\bigr)$.  The proofs will be more cumbersome
but the rates and the limit distributions of the studied estimators will be
the same.

Remark that the same mathematical models are used in the applications related
to the detection of weak optical signals from two sources.

Of course, it will be interesting to see the conditions of identifiability in
the situation where the beginning of the emissions of these sources are
unknown and have to be estimated together with the positions.  In the case of
one source such problem was discussed in the work~\cite{CDFK22}.

All useful information about the position of sources, according to the
statements of this work, is contained in the times of arrival of the signals
$\tau_{1,k}(\vartheta)$, $\tau_{2,k}(\vartheta)$, $k=1,\ldots,K$.  It is
possible to study another statement of the problem supposing that these
moments are estimated separately by observations $X_k=(X_k(t),\ 0\leq t\leq
T)$ for each $k$, say by $\hat\tau_{i,k,n} = \hat\tau _{i,k,n}(X_k)$, and then
the positions of the sources are estimated on the base of the obtained
estimators $\hat\tau _{i,k,n}$, $i=1,2$, $k=1,\ldots,K$.  Such approach was
considered in the works~\cite{CDFK22,CK20} (see as well~\cite{Kut22}).

Note that in the works on Poisson source localization, the intensities of
Poisson processes are sometimes taken in the form
\[
\lambda_k(\vartheta,t) = F\bigl(\rho_k(\vartheta)\bigr)\, S_k(t) + \lambda_0
,\qquad 0\leq t\leq T,
\]
where $F(\cdot)$ is a known strictly decreasing function of the distance
$\rho_k$ between the source and the $k$-th detector.  The developed in the
present work approach (smooth case) can be applied for such models too,
because here all the useful information is contained in the distances
$\rho_k(\vartheta)$.

\bigskip
\noindent
\textbf{Acknowledgments.}
This research was financially supported by the Ministry of Education and
Science of Russian Federation (project No.~FSWF-2023-0012) for sections~2.1
and~2.2, and by the Russian Science Foundation (project No.~20-61-47043) for
sections~2.3 and~2.4.

\bigskip
\noindent
\textbf{Conflict of interest.}
On behalf of all authors, the corresponding author states that there is no
conflict of interest.

\end{document}